\documentclass[a4paper,11pt]{article}
\usepackage{amssymb,latexsym,amscd,array}
\usepackage{exscale}
\usepackage[centertags]{amsmath}
\usepackage{amssymb}
\usepackage{amsthm}
\usepackage{dsfont}
\usepackage[all]{xy}
\usepackage{lscape}
\usepackage{pdflscape}
\usepackage{color}
\usepackage{tikz}
\usetikzlibrary{calc,shapes}

\newtheorem{thm}{Theorem}[section]
\newtheorem{prop}[thm]{Proposition}
\newtheorem*{ex}{Example}
\newtheorem{rem}[thm]{Remark}
\newtheorem{lem}[thm]{Lemma}
\newtheorem{defn}[thm]{Definition}
\newtheorem{cor}[thm]{Corollary}

\newcommand{\mbc}{\mathbb{C}}
\newcommand{\mbd}{\mathbb{D}}
\newcommand{\mbl}{\mathbb{L}}
\newcommand{\mbn}{\mathbb{N}}
\newcommand{\mbp}{\mathbb{P}}
\newcommand{\mbq}{\mathbb{Q}}
\newcommand{\mbr}{\mathbb{R}}
\newcommand{\mbv}{\mathbb{V}}
\newcommand{\mbz}{\mathbb{Z}}

\newcommand{\dC}{{\mathbb C}}

\newcommand{\dZ}{{\mathbb Z}}

\newcommand{\dL}{{\mathbb L}}

\newcommand{\ra}{\rightarrow}
\newcommand{\lra}{\longrightarrow}

\newcommand{\mcb}{\mathcal{B}}

\newcommand{\mcd}{\mathcal{D}}

\newcommand{\mch}{\mathcal{H}}

\newcommand{\mcl}{\mathcal{L}}
\newcommand{\mcm}{\mathcal{M}}
\newcommand{\mcn}{\mathcal{N}}
\newcommand{\mco}{\mathcal{O}}

\newcommand{\mcr}{\mathcal{R}}
\newcommand{\mcs}{\mathcal{S}}

\newcommand{\cD}{\mathcal{D}}

\newcommand{\cM}{\mathcal{M}}

\newcommand{\p}{\partial}

\DeclareMathOperator{\Spec}{\textup{Spec}\,}

\DeclareMathOperator{\FL}{\textup{FL}}

\numberwithin{equation}{section}
\begin{document}
\title{Laurent Polynomials, GKZ-hypergeometric Systems and\\ Mixed Hodge Modules}
\author{Thomas Reichelt}
\date{\today}
\maketitle

\begin{abstract}
We endow certain GKZ-hypergeometric systems with a natural structure of a mixed Hodge module, which is compatible with the mixed Hodge module structure on the Gau\ss-Manin system of an associated family of Laurent polynomials. As an application we show that the underlying perverse sheaf of a GKZ-system with rational parameter has quasi-unipotent local monodromy.
\end{abstract}

\renewcommand{\thefootnote}{}
\footnote{
2010 \emph{Mathematics Subject Classification.}
14D07, 32S35, 32S40, 32G34, \\
Keywords: Gau\ss-Manin system, hypergeometric $\cD$-module, Radon transformation, mixed Hodge module\\
The author was supported by a postdoctoral fellowship of the ``Fondation sciences math\'ematiques de Paris'' \\
and received partial support by the ANR grant ANR-08-BLAN-0317-01 (SEDIGA).
}

\section*{Introduction}
At the end of the 80's Gelfand, Kapranov and Zelevinsky introduced differential equations, which are a vast generalization of Gau\ss's hypergeometric equation and which are nowadays called GKZ-hypergeometric systems. Since then, GKZ-systems found applications in many fields of mathematics like representation theory, combinatorics and in particular in the area of mirror symmetry. 

The initial data for a GKZ-system is an integer matrix $A$ together with a parameter vector $\beta$. It is regular singular if and only if it is homogeneous, i.e. if the vector $(1,\ldots ,1)$ is in the row span of $A$. For a fixed matrix $A$ the structure of the GKZ-system depends heavily on the parameter $\beta$. The easiest case is when $\beta$ satisfies a so-called non-resonance condition. Then it was shown in \cite{GKZ1} that the solution complex of the corresponding GKZ-system is isomorphic to a direct image of a local system, defined on the complement of
the graph of an associated family of Laurent polynomials, under the projection to the parameter space.
In another direction, Adolphson and Sperber \cite{AdoSperberpubl} showed that a non-resonant, homogeneous GKZ-system is isomorphic to a direct factor of a Gau\ss-Manin system of a family of affine varieties constructed from a related family of Laurent polynomials.

However, the GKZ-systems which appear in mirror symmetry are usually resonant. It was found by Batyrev \cite{Bat4} that the periods of some families of Calabi-Yau hypersurfaces in toric varieties are among the solutions of these GKZ-systems. He achieved this by studying the variation of mixed Hodge structures of the complement of the corresponding affine varieties inside the dense torus of the toric variety. The results of Batyrev were refined by Stienstra in \cite{Sti}, where he proved that the relative cohomology bundle of this affine family is actually isomorphic to a GKZ-system outside its singular locus. This result of Stienstra endows the smooth part of a GKZ-system with a (geometric) variation of mixed Hodge structures.

The proof of \cite{GKZ1} consists, among other things, in showing that the total Fourier-Laplace transformation of a non-resonant GKZ-system, which has support on an affine (generalized) toric variety, is isomorphic to the middle extension of a rank one bundle which is defined on the dense torus of this variety. This was generalized by Schulze and Walther in \cite{SchulWalth2}. Using the theory of Euler-Koszul complexes, which was developed in \cite{MillerWaltherMat}, they identified a set of so-called non-strongly-resonant parameters (including the non-resonant ones), for which the total Fourier-Laplace transformation of the corresponding GKZ-systems is isomorphic to the direct image of the rank one bundle under the torus embedding. 
This result turns out to be crucial for our study, as we prove that such GKZ-systems with integer parameter (which are resonant by definition) carry a mixed Hodge module structure, which extends the variation of mixed Hodge structures on the smooth part found by Batyrev and Stienstra. 

In order to prove this, we use a comparison theorem of d'Agnolo and Eastwood \cite{AE} between the total Fourier-Laplace transformation and various so-called Radon transformations to show that the GKZ-system can be written as a certain Radon transformation of a mixed Hodge module.  But this endows the GKZ-system with the structure of a mixed Hodge module, because the Radon transformations (unlike the total Fourier Laplace transformation) preserve the category of mixed Hodge modules.

The benefit of this approach is that we can show that the GKZ-system endowed with this Hodge structure sits inside an exact $4$-term sequence of mixed Hodge modules where the two outer terms are constant variations of mixed Hodge structures and the second term is isomorphic to the Gau\ss-Manin system of an associated family of Laurent polynomials. This establish a tight relationship between this Gau\ss-Manin system and the GKZ-system which turns out to be very useful in computing various Landau-Ginzburg models in mirror symmetry (cf. \cite{ReiSe} and \cite{ReiSe2}). 

Let us give a short overview of the paper: In the first section we review the definition of GKZ-systems and the theorem of Schulze and Walther which expresses the GKZ-system as a total Fourier-Laplace transformation of a direct image of a rank one bundle, when the parameter $\beta$ is not strongly resonant. We first prove some basic facts on the geometry of the set of strongly resonant parameters $sRes(A)$ for general matrices $A$ and also in the easier case when the associated semi-group $\mbn A$ is saturated. By using a result of Walther \cite{Walther1} on the holonomic dual of a GKZ-system and results of Saito \cite{SaitoMut1} on their classification, we are able to prove a result which is in some sense dual to the one of Schulze and Walther, namely we identify the parameters for which the total Fourier-Laplace transformation of a GKZ-system is isomorphic to the \textbf{proper} direct image of the rank one bundle under the torus embedding.

In the second section a tight relation between certain (direct sums of) GKZ-systems and Gau\ss-Manin systems of associated families of Laurent polynomials is established (Theorem \ref{thm:genseq} and Corollary \ref{cor:exseqA}). More precisely, we show that there exists a morphism between the Gau\ss-Manin system and the GKZ-system with an $\mco$-free kernel and cokernel. Using a description of the Gau\ss-Manin system by relative differential forms , we explicitly compute this morphism, in the case when the semigroup $\mbn A$ is saturated, which gives interesting insight into the structure of these Gau\ss-Manin systems. Finally, when the matrix $A$ satisfies some extra homogenity condtion, we restrict these systems to a  hyperplane and recover the above-mentioned result of \cite[Theorem 8]{Sti}. 

In the last section we show that the results in the second section carry over into the category of mixed Hodge modules (Proposition \ref{prop:MHMgenseq} and Proposition \ref{prop:MHMexseqA}). This enables us to prove that a homogeneous GKZ-system with a non strongly-resonant, integer parameter vector $\beta$ carries a mixed Hodge module structure. If the parameter vector $\beta$ is rational we show that the GKZ-system is a direct summand in a mixed Hodge module, which shows that the underlying perverse sheaf has quasi-unipotent local monodromy. 

\textbf{Acknowledgements:}  I would like to thank Hiroshi Iritani, Etienne Mann, Thierry Mignon and Christian Sevenheck for useful discussions. Furthermore, I thank Claus Hertling and Claude Sabbah for their continuous support and interest in my work. I am indebted to Mutsumi Saito  who pointed out to me that the proof of Proposition \ref{prop:gkzuniquemorph} follows directly from his work \cite{SaitoMut3}. I am also grateful to Uli Walther who answered me some questions related to his joint work \cite{SchulWalth2} with Mathias Schulze.\\

\section{GKZ-systems}

In this section we will prove various facts about GKZ-systems. An important ingredient, which we will need in the next section, is
a theorem of Schulze and Walther \cite{SchulWalth2} which expresses the Fourier-Laplace transformation of a GKZ-system as a direct image of a rank one bundle under a torus-embedding, when the parameter vector $\beta$ is not in the set of so-called strongly resonant values. We will prove some basic facts about this set which will be essential in the following. By combing results of Walther \cite{Walther1} on the holonomic dual of a GKZ-system for generic parameter $\beta$ and of Saito \cite{SaitoMut1} on isomorphism classes of GKZ-systems, we are able to prove in Proposition \ref{prop:dualSW} the dual statement of the theorem of Schulze and Walther.

Let $X$ be a smooth algebraic variety over $\mbc$ of dimension $n$.
We denote by $M(\mcd_X)$ the abelian category of algebraic $\mcd_X$-modules on $X$ and the abelian subcategory of (regular) holonomic $\mcd_X$-modules by $M_h(\mcd_X)$ (resp. $(M_{rh}(\mcd_X))$.  The full triangulated subcategory in $D^b(\mcd_X)$, consisting of objects with (regular) holonomic cohomology, is denoted by $D^b_{h}(\mcd_X)$ (resp. $D^b_{rh}(\mcd_X)$).

\noindent Let $f: X \ra Y$ be a map between smooth algebraic varieties. Let $M \in D^b(\mcd_X)$ and $N \in D^b(\mcd_Y)$, then we denote by $f_+ M := Rf_*(\mcd_{Y \leftarrow X} \overset{L}{\otimes} M)$ resp. $f^+ M:= \mcd_{X \ra Y} \overset{L}{\otimes} f^{-1}M$ the direct resp. inverse image for $\mcd$-modules.
Recall that the functors $f_+,f^+$ preserve (regular) holonomicity (see e.g., \cite[Theorem 3.2.3]{Hotta}).
\noindent We denote by $\mbd: D^b_h(\mcd_X) \ra (D^b_h(\mcd_X))^{opp}$ the holonomic duality functor.
Recall that for a single holonomic $\mcd_X$-module $M$, the holonomic dual is also a single holonomic $\mcd_X$-module (\cite[Proposition 3.2.1]{Hotta}) and that holonomic duality preserves regular holonomicity ( \cite[Theorem 6.1.10]{Hotta}).

\noindent For a morphism $f: X \ra Y$ between smooth algebraic varieties we additionally define the functors $f_\dag := \mbd \circ f_+ \circ \mbd$ and $f^\dag := \mbd \circ f^+ \circ \mbd$.

\begin{defn}[\cite{GKZ1}, \cite{Adolphson}]
\label{def:GKZ}
Consider a lattice $\dZ^{d}$ and vectors $\underline{a}_1,\ldots,\underline{a}_n\in\dZ^{d}$ which we also write as
a matrix $A=(\underline{a}_1,\ldots,\underline{a}_n)$. In the following we assume that the vectors $\underline{a}_1, \ldots , \underline{a}_n$ generate $\mbz^d$ as a $\mbz$-module. Moreover, let $\beta=(\beta_1,\ldots,\beta_d)\in\dC^{d}$.
Write $\dL$ for the $\mbz$-module of integer relations among the columns of $A$ and
$\cD_{\dC^{n}}$ for the sheaf of rings of algebraic differential operators on $\dC^{n}$ (where
we choose $\lambda_1,\ldots,\lambda_n$ as coordinates).
Define
$$
\cM^\beta_A:=\cD_{\dC^n}/\left((\Box_{\underline{l}})_{\underline{l}\in\dL}+(E_k - \beta_k)_{k=1,\ldots d}\right),
$$
where
$$
\begin{array}{rcl}
\Box_{\underline{l}} & := & \prod_{i:l_i<0} \partial_{\lambda_i}^{-l_i} -\prod_{i:l_i>0} \partial_{\lambda_i}^{l_i}\, , \\ \\
E_k & := & \sum_{i=1}^n a_{ki} \lambda_i \partial_{\lambda_i}\, .
\end{array}
$$
$\cM^\beta_A$ is called a GKZ-system. It is a holonomic $\mcd$-module by \cite[Theorem 3.9]{Adolphson}. 
\end{defn}

\noindent As GKZ-systems are defined on the affine space $\mbc^{n}$, we will often work with the $D$-modules of global sections $M^\beta_A := \Gamma(\mbc^{n} , \mcm^\beta_A)$ rather than with the sheaves themselves. Here $D$ denotes the Weyl algebra $\mbc[\lambda_1, \ldots , \lambda_n]\langle \p_{\lambda_1}, \ldots , \p_{\lambda_n} \rangle$.

\begin{ex}[I]
The easiest examples of GKZ-systems are the ones associated to the matrix $A = (1)$. These are simply the $\mcd$-modules
\[
\mcd_\mbc / \mcd_\mbc(\lambda \p - \beta)\, .
\]
If we assume $\beta \in \mbz$, then it is easy to compute, that the left-action of $\p_\lambda$ is invertible if and only if $\beta \geq 0$. The generalization of this property will be shown in Theorem \ref{thm:SW} and Lemma \ref{lem:sResnormal}. Furthermore the holonomic dual of $\mcd_\mbc / \mcd_\mbc(\lambda \p - \beta)$ is $\mcd_\mbc / \mcd_\mbc(\lambda \p + \beta+1)$, the corresponding statement for general GKZ-systems is given in Proposition \ref{prop:dualGKZ}. 
\end{ex}

\noindent Denote by
$$
\mbn A := \{ \sum_{i=1}^d \gamma_i \underline{a}_i \in \mbz^d \mid (\gamma_i)_{i = 1, \ldots ,d} \in \mbn^d\}
$$
the semigroup built by the columns of $A$ seen as elements in $\mbz^d$ and similarly for $\mbz A$ and $\mbr_+ A$, where $\mbr_+$ are the real numbers bigger or equal than zero.

\noindent The semigroup ring associated to the matrix $A$ is $S_A := \mbc[\mbn A] \simeq R/I_A$, where $R$ is the commutative ring $\mbc[\p_{\lambda_1}, \ldots, \p_{\lambda_n}]$, $I_A$ is the ideal
\[
I_A = \{ \Box_{\underline{l}}, l \in \mbl\}
\]
and the isomorphism follows from \cite[Theorem 7.3]{MillSturm}.
The rings $R$ and $S_A$ are naturally $\mbz^{d}$-graded if we define $\deg(\p_{\lambda_j}) = \underline{a}_j$ for $j=1, \ldots , n$. This is compatible with the $\mbz^{d}$-grading of the Weyl algebra $D$ given by $\deg(\lambda_j) = - \underline{a}_j$ and $\deg(\p_{\lambda_j}) = \underline{a}_j$.

\begin{ex}[II]
Consider the matrix
\[
A = \left(\begin{matrix}3 & 2 & 0 \\ 1 & 1 & 1 \end{matrix} \right)
\]
A simple Gr\"{o}bner basis computation shows that the ideal $I_A$ is generated by the single box-operator $\Box_{(3,-5,2)}$, i.e we have
\[
S_A \simeq \mbc[\p_1,\p_2,\p_3]/(\p_2^5-\p_1^3\cdot \p_3^2).
\]
The $\mbz^2$-degrees of $S_A$ can be read off from the following diagram:
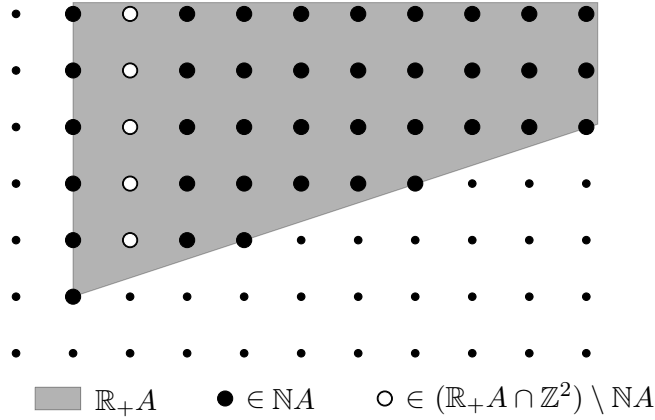
\begin{center}
\begin{tikzpicture}
 \foreach \x in {-1,0,...,9}{
      \foreach \y in {-1,0,...,5}{ 
        \node[draw,circle,inner sep=1pt,fill] at (0.75*\x,0.75*\y) {};
      }
    }

 \filldraw[black,opacity=.3] (0,0) -- (0.75*9.2,0.75*3.05) -- (0.75*9.2,0.75*5.2) -- (0,0.75*5.2) -- (0,0); 
 
 \foreach \y in {3,4,...,5}{
   \foreach \x in{0,1,...,9}{
    \node[draw,circle,inner sep=2pt,fill] at (0.75*\x,0.75*\y) {};
   }
  }
\foreach \x in {0,1,...,6}{
    \node[draw,circle,inner sep=2pt,fill] at (0.75*\x,0.75*2) {};
  }
\foreach \x in {0,1,...,3}{
    \node[draw,circle,inner sep=2pt,fill] at (0.75*\x,0.75*1) {};
  }
\foreach \y in {1,2,...,5}{
    \node[draw,circle,inner sep=1.5pt,fill,white] at (0.75*1,0.75*\y) {};
  }
\node[draw,circle,inner sep=2pt,fill] at (0.75*0,0.75*0) {};
\filldraw[black,opacity=.3] (-0.5,-1.2) -- (0.1,-1.2) -- (0.1,-1.5) -- (-0.5,-1.5) -- (-0.5,-1.2); 
\node[] at (0.7,-1.41) {$\mbr_+ A$};
\node[draw,circle,inner sep=2pt,fill] at (2.0,-1.35) {};
\node[] at (2.4,-1.37) {$\in$};
\node[] at (2.9,-1.35) {$\mbn A$};
\node[draw,circle,inner sep=2pt,fill] at (4.1,-1.35) {};
\node[draw,circle,inner sep=1.5pt,fill,white] at (4.1,-1.35) {};
\node[] at (4.5,-1.37) {$\in$};
\node[] at (6.2,-1.35) {$(\mbr_+ A \cap \mbz^2) \setminus \mbn A$};
\end{tikzpicture}

\end{center}
Notice that $\mbn A = \mbr_+A \cap \mbz^2$ (condition \eqref{eq:defsaturated}) is not satisfied.
\end{ex}
\noindent If we substitute $\p_{\lambda_i} \mapsto \mu_i$, the ideal $I_{A}$ goes over to an ideal $\hat{I}_{A}\subset \mbc[\mu_1, \ldots , \mu_n]$ and we have clearly the following isomorphism
\[
\mbc[\mbn A] \simeq \mbc[\mu_1, \ldots , \mu_n] / \hat{I}_{A}\, .
\]
Furthermore, we have a ring homomorphism
\begin{align}
\mbc[\mu_1, \ldots , \mu_n] / \hat{I}_{A} &\lra \mbc[y_1^\pm, \ldots , y_d^\pm]\, , \notag \\
\mu_i &\mapsto \underline{y}^{\underline{a}_i}:= \prod_{k=1}^d y_k^{a_{ki}}\, . \notag
\end{align}
If we define $Y':= \Spec(\mbc[\mbn A])$ and $T' := \Spec(\mbc[y_1^\pm, \ldots , y_d^\pm])$, the ring homomorphism above gives rise to the following map
\begin{equation}\label{eq:bigtorusembedding}
h_A : T' \lra \mbc^n\, ,
\end{equation}
where $Y'$ is the closure of the image of $h_A$.

Let $W$ be a finite-dimensional complex vector space and $W'$ its dual vector space with coordinates $\underline{\lambda} =(\lambda_1, \ldots ,\lambda_n)$ resp. $\underline{\mu} =(\mu_1, \ldots , \mu_n)$. Let $\langle \cdot, \cdot \rangle$ be the standard euclidean pairing with respect to these coordinates.

\begin{defn}
Define $\mcl := \mco_{W' \times W} e^{-\langle \underline{\mu}, \underline{\lambda} \rangle}$ and denote by  $p_1: W' \times W \ra W'$, $p_2 : W' \times W \ra W$ the canonical projections. The Fourier-Laplace transformation is then defined by
\[
\FL(M) := p_{2 +}(p_1^+ M \overset{L}{\otimes} \mcl) \quad M \in D^b_h(\mcd_{W'}).
\]
\end{defn}
In general, the Fourier-Laplace transformation does not preserve regular holonomicity. But for the derived category of complexes of $\mcd$-modules with so-called monodromic cohomology, regular holonomicity is preserved. Let $\theta: \mbc^\ast \times W' \ra W'$ be the natural $\mbc^\ast$ action on $W'$ and let $z$ be a coordinate on $\mbc^\ast$. We denote the push-forward $\theta_*(z \partial_z)$ as the Euler vector field $\mathfrak{E}$.\\

\begin{defn}\cite{Brylinski}
A regular holonomic $\mcd_{W'}$-module $M$ is called monodromoic, if the Euler field $\mathfrak{E}$ acts finitely on the global sections of $M$, i.e. for a global section section $v$ of $M$ the set $\mathfrak{E}^n(v)$, $(n \in \mbn)$, generates a finite-dimensional vector space.
We denote by $D^b_{mon}(\mcd_{W'})$ the derived category of bounded complexes of $\mcd_{W'}$-modules with regular holonomic and monodromic cohomology. 
\end{defn}

\begin{thm}\cite{Brylinski}
\begin{enumerate}
\item $\FL$ preserves complexes with monodromic cohomology.
\item In $D^b_{mon}(\mcd_{W'})$ we have
\[
\FL \circ \FL \simeq Id \quad \text{and} \quad \mbd \circ \FL \simeq \FL \circ \mbd \, .
\]
\item $\FL$ is $t$-exact with respect to the natural $t$-structure on $D^b_{mon}(\mcd_{W'})$ resp. $D^b_{mon}(\mcd_{W})$.
\end{enumerate}
\end{thm}
\begin{proof}
The above statements are stated in \cite{Brylinski} for constructible monodromic complexes. One has to use the Riemann-Hilbert correspondence, \cite[Proposition 7.12, Theorem 7.24]{Brylinski} to translate the statements. So the first statement is Corollaire 6.12, the second statement is Proposition 6.13 and the third is Corollaire 7.23 in \cite{Brylinski}.
\end{proof}

\begin{rem}\label{rem:mon-dmod}
By Proposition 7.12 of \cite{Brylinski} a complex in $D^b_{mon}(\mcd_{W'})$ corresponds to a complex of sheaves of $\mbc$-vector spaces with constructible, monodromic cohomology, where we call a sheaf of $\mbc$-vector spaces monodromic if it is locally constant along the orbits of the $\mbc^*$-action on the fibers of $W'$. Let $\pi_{W'}: W' \setminus 0 \ra \mbp(W')$ be the projectivization of $W'$. This characterization shows that a complex $M \in D^b_{mon}(\mcd_{W'})$ is monodromic if its restriction to the complement of zero is isomorphic to $\pi^+_{W'} N$ for some $N \in D^b_{rh}(\mcd_{\mbp(W')})$.
\end{rem}

\noindent Because of the discussion above the hypergeometric system $\mcm^\beta_{A}$ can be seen as the total Fourier-Laplace transform of a $\mcd$-module which has support on $Y'$. In \cite{SchulWalth2} it it shown that
$\textup{FL}(\mcm^\beta_{A})$ is isomorphic to $h_{A,+} \mco_{T'}\underline{y}^\beta$ for certain $\beta \in \mbc^d$, where we denote by $\mco_{T'}\underline{y}^\beta$ the $\mcd$-module
$$
\mcd_{T'}/\mcd_{T'} \cdot (y_1 \p_{y_1} - \beta_1, \ldots , y_d \p_{y_d} - \beta_d)\, .
$$

Notice that GKZ-systems are not monodromic in general. However, in Lemma \ref{lem:GKZmono} we will give a sufficient condition for the matrix $A$ in order that the corresponding GKZ-system is monodromic.

\begin{defn}[\cite{MillerWaltherMat} Definition 5.2]
Let $N$ be a finitely generated $\mbz^{d}$-graded $R$-module. An element
$\alpha \in \mbz^{d}$ is called a true degree of $N$ if $N_\alpha$ is non-zero. A vector $\alpha \in \mbc^{d}$ is called a quasi-degree of $N$, written $\alpha \in qdeg(N)$, if
$\alpha$ lies in the complex Zariski closure $qdeg(N)$ of the true degrees of $N$ via the natural embedding $\mbz^{d} \hookrightarrow \mbc^{d}$.  
\end{defn} 

\noindent Schulze and Walther now define the following set of parameters:
\begin{defn}[\cite{SchulWalth2}] 
The set
\[
sRes(A) := \bigcup_{j=1}^n sRes_j(A)
\]
is called the set of strongly resonant parameters of $A$, where
\[
sRes_j(A) := \{ \beta \in \mbc^{d} \mid \beta \in -(\mbn +1)\underline{a}_j - qdeg(S_A / \langle \p_{\lambda_j} \rangle)\} .
\]
\end{defn}
Notice that we use the convention $deg(\p_j) = \underline{a}_j$ instead of $deg(\p_j) = -\underline{a}_j $ as in \cite{SchulWalth2}, because this leads to simpler formulas.
\begin{rem}
The set $sRes(A)$ is a proper subset of the set of resonant parameters, which was introduced in \cite{GKZ1} and which is defined by
\[
Res(A) := \bigcup_{F \; \text{face of}\; A} \mbz^d + \mbc F \; ,
\]
where a face $F$ of $A$ is a set of columns of $A$ minimizing some linear functional on $\mbn A \subset \mbz^{d}$ and $\mbc F$ is the $\mbc$-linear span of $F$.
\end{rem}
\noindent The condition on $\beta$ for $\FL(\mcm_A^\beta)$ being isomorphic to $h_{A,+} \mco_{T'}\underline{y}^\beta$ is now the following:
\begin{thm}[\cite{SchulWalth2}\label{thm:SW} Theorem 3.6, Corollary 3.7]\label{thm:GKZisoSW} Let $\mbn A$
be a positive semigroup, meaning that $0$ is the only unit in $\mbn A$.
Then the following is equivalent
\begin{enumerate}
\item $\beta \notin sRes(A)$.
\item $\mcm^\beta_A \simeq \FL ((l \circ k)_+ \, \mco_{T'}\underline{y}^\beta)$.
\item Left multiplication with $\p_{\lambda_i}$ is invertible on $M^{\beta}_{A}$ for $i= 1, \ldots ,n$.
\end{enumerate}
\end{thm}

\noindent We now want to characterize the set $sRes(A)$. For this we have to understand the geometry of the sets $qdeg(S_A/ \langle \p_{\lambda_j} \rangle)$ for $j = 1,\ldots, n$. We use the following definitions and notations of \cite{MillerWaltherMat}.  Let $S_F$ be the semigroup ring generated by a face $F$, then a $\mbz^{d}$-graded $R$-module $M$ is called toric if it has a toric filtration
$$
0 = M_0 \subset M_1 \subset \ldots  \subset M_{l-1} \subset M_l = M\, ,
$$
meaning that, for each $k$, $M_k/M_{k-1}$ is a $\mbz^{d}$-graded translate of $S_{F_k}$ for some face $F_k$ of $\mbn A$, generated in degree $\alpha_k$, which will be denoted by $S_{F_k}(\alpha_k)$. 

\noindent Now by example $4.7$ in \cite{MillerWaltherMat} one easily deduces that the $\mbz^{d}$-graded rings $S_A/ \langle \p_j \rangle$ are toric for $j = 1, \ldots , n$.\\

\noindent  The toric filtration
\begin{equation}\label{toricfiltj}
0 =M_0^j \subset M_1^j \subset \ldots \subset M_{l_j}^j = S_A / \langle \p_{\lambda_j} \rangle
\end{equation}
with $M^j_k/M^j_{k-1} \simeq S_{F_{k,j}}(\alpha_{k,j})$
gives us the following decomposition of the degrees 
$$
deg(S_A / \langle \p_{\lambda_j} \rangle) = \bigcup_{k=1}^{l_j} deg(S_{F_{k,j}}(\alpha_{k,j})) 
$$
resp. of the quasidegrees 
\begin{equation}\label{eq:quasideg}
qdeg(S_A / \langle \p_{\lambda_j} \rangle) = \bigcup_{k=1}^{l_j} qdeg(S_{F_{k,j}}(\alpha_{k,j})) 
\end{equation}
of $S_A/ \langle \p_{\lambda_j} \rangle$, where only faces $S_{F_{k,j}}$ occur which do not contain $\underline{a}_j$. The quasi-degrees of an $\mbz^d$-graded $R$-module $S_F$, where $F$ is a proper face of $\mbn A$, are just the $\mbc$-linear span of $F$, which we denote by $\mbc F$. (Here we use the embedding $\mbz^d \hookrightarrow \mbc^d$).
Thus the quasi-degrees of $S_A/\langle \p_{\lambda_j} \rangle$ are just finite unions of translates of the linear subspaces $\mbc F_{k,j}$.\\

The following lemma shows that a translate of the cone $\mbr_+ A$ does not meet the set of strongly resonant parameters. (Notice that the following lemma with $\delta_A = 0$ is stated in Corollary 3.8 of \cite{SchulWalth2}, but not proven. An easy counterexample of their claim is provided by $A= (2,5)$.) 

\begin{lem}\label{lem:sResnonnormal}
Denote as above by $sRes(A)$ the set of strongly resonant vectors $\beta$. 
There exists $\delta_A \in \mbn A$ such that
\[
(\mbr_+ A + \delta_A) \; \cap\; sRes(A) = \emptyset
\]
\end{lem}
\begin{proof}
Choose some non-zero vector $\alpha'$ in the interior of $\mbn A$, i.e. in $\mbn A \cap (\mbr_+ A)^\circ$. As we assumed that $\mbn A$ is a positive semigroup, the set $\mbr_+ A$ is a strongly convex cone in $\mbr^d$. We noticed above that $qdeg(S_{F})$ is the $\mbc$-linear span of the face $F$, therefore we can conclude that
$$
\mbr_+ A + \alpha' + \alpha_{k,j} \cap qdeg(S_{F_{k,j}}(\alpha_{k,j})) = \emptyset
$$
where we have used $qdeg(S_{F_{k,j}}(\alpha_{k,j})) = qdeg(S_{F_{k,j}}) + \alpha_{k,j} $. As every $\alpha_{j,k}$ for $j=1,\ldots ,d$ and $k=1, \ldots, l_j$ lies in $\mbn A$, the element $\alpha'' := \sum_{j=1}^d \sum_{k=1}^{l_j} \alpha_{k,j}$ also lies in $\mbn A$.  
Now set $\delta_A := \alpha' + \alpha''$. We have $\mbr_+A + \delta_A \subset \mbr_+ A + \alpha' + \alpha_{k,j}$ and therefore
$$
\mbr_+ A + \delta_A \cap qdeg(S_A / \langle \p_{\lambda_j} \rangle) = \emptyset\, .
$$
As $-a_j$ lies in $- \mbn A$, this shows that 
$$
\mbr_+ A + \delta_A  \cap qdeg(S_A / \langle \p_{\lambda_j} \rangle) - (l+1)a_j = \emptyset
$$ 
for every $l \in \mbn$ and therefore $\mbr_+ A + \delta_A \cap sRes_j(A) = \emptyset$. As this is true for every $j=1, \ldots ,d$, this shows the claim.
\end{proof}

\begin{ex}[II continued]
We continue our study of the GKZ-systems associated to the matrix 
\[
A = \left(\begin{matrix}3 & 2 & 0 \\ 1 & 1 & 1 \end{matrix}\right)
\]
and the corresponding strongly resonant values $sRes(A)$. The degrees and quasi-degrees of $S_A / \langle \p_i \rangle$ are sketched below. 
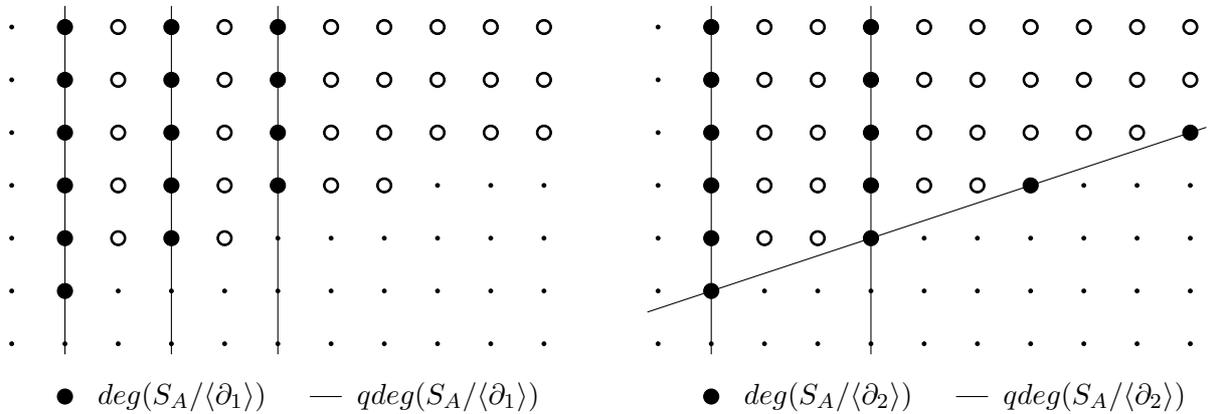
\begin{center}
\newdimen\scale
\scale=0.7cm
\begin{tikzpicture}
 \foreach \x in {-1,0,...,9}{
   \foreach \y in {-1,0,...,5}{ 
     \node[draw,circle,inner sep=0.5pt,fill] at (\scale*\x,\scale*\y) {};
   }
 }
 \foreach \x in {-1,0,...,9}{
      \foreach \y in {-1,0,...,5}{ 
        \node[draw,circle,inner sep=0.5pt,fill] at (8.5cm +\scale*\x,\scale*\y) {};
      }
    }
\foreach \y in {0,1,...,5}{
  \node[draw,circle,inner sep=2pt,fill] at (\scale*0,\scale*\y) {};
}
\foreach \y in {1,2,...,5}{
  \node[draw,circle,inner sep=2pt,fill] at (\scale*2,\scale*\y) {};
}
\foreach \y in {2,3,...,5}{
  \node[draw,circle,inner sep=2pt,fill] at (\scale*4,\scale*\y) {};
}
\foreach \y in {1,2,...,5}{
  \node[draw,circle,inner sep=2pt,fill] at (\scale*1,\scale*\y) {};
  \node[draw,circle,inner sep=1.3pt,fill,white] at (\scale*1,\scale*\y) {};
\node[draw,circle,inner sep=2pt,fill] at (\scale*3,\scale*\y) {};
  \node[draw,circle,inner sep=1.3pt,fill,white] at (\scale*3,\scale*\y) {};
}
\foreach \y in {2,3,...,5}{
  \node[draw,circle,inner sep=2pt,fill] at (\scale*5,\scale*\y) {};
  \node[draw,circle,inner sep=1.3pt,fill,white] at (\scale*5,\scale*\y) {};
  \node[draw,circle,inner sep=2pt,fill] at (\scale*6,\scale*\y) {};
  \node[draw,circle,inner sep=1.3pt,fill,white] at (\scale*6,\scale*\y) {};
}
\foreach \x in {7,8,...,9}{
  \foreach \y in {3,4,...,5}{
  \node[draw,circle,inner sep=2pt,fill] at (\scale*\x,\scale*\y) {};
  \node[draw,circle,inner sep=1.3pt,fill,white] at (\scale*\x,\scale*\y) {};
  }
}
\draw (\scale*0,\scale* -1.2) -- (\scale*0, \scale *5.4);
\draw (\scale*2,\scale* -1.2) -- (\scale*2, \scale *5.4);
\draw (\scale*4,\scale* -1.2) -- (\scale*4, \scale *5.4);
\node[draw,circle,inner sep=2pt,fill] at (\scale*0,\scale*-2) {};
\node[] at (\scale*2.2,\scale*-2) {$deg(S_A/ \langle \p_1 \rangle)$};
\draw (\scale*4.6,\scale* -2) -- (\scale*5.2, \scale *-2);
\node[] at (\scale*7.2,\scale*-2) {$qdeg(S_A/ \langle \p_1 \rangle)$};

\foreach \y in {0,...,5}{
\node[draw,circle,inner sep=2pt,fill] at (8.5cm+\scale*0,\scale*\y) {};
}
\foreach \y in {1,...,5}{
\node[draw,circle,inner sep=2pt,fill] at (8.5cm+\scale*3,\scale*\y) {};
}
\node[draw,circle,inner sep=2pt,fill] at (8.5cm+\scale*6,\scale*2) {};
\node[draw,circle,inner sep=2pt,fill] at (8.5cm+\scale*9,\scale*3) {};
\foreach \x in {1,...,2}{
  \foreach \y in {1,...,5}{
  \node[draw,circle,inner sep=2pt,fill] at (8.5cm+\scale*\x,\scale*\y) {};
  \node[draw,circle,inner sep=1.3pt,fill,white] at (8.5cm+\scale*\x,\scale*\y) {};
  }
}
\foreach \x in {4,...,8}{
  \foreach \y in {3,...,5}{
  \node[draw,circle,inner sep=2pt,fill] at (8.5cm+\scale*\x,\scale*\y) {};
  \node[draw,circle,inner sep=1.3pt,fill,white] at (8.5cm+\scale*\x,\scale*\y) {};
  }
}
\node[draw,circle,inner sep=2pt,fill] at (8.5cm+\scale*4,\scale*2) {};
\node[draw,circle,inner sep=1.3pt,fill,white] at (8.5cm+\scale*4,\scale*2) {};
\node[draw,circle,inner sep=2pt,fill] at (8.5cm+\scale*5,\scale*2) {};
\node[draw,circle,inner sep=1.3pt,fill,white] at (8.5cm+\scale*5,\scale*2) {};
\node[draw,circle,inner sep=2pt,fill] at (8.5cm+\scale*9,\scale*4) {};
\node[draw,circle,inner sep=1.3pt,fill,white] at (8.5cm+\scale*9,\scale*4) {};
\node[draw,circle,inner sep=2pt,fill] at (8.5cm+\scale*9,\scale*5) {};
\node[draw,circle,inner sep=1.3pt,fill,white] at (8.5cm+\scale*9,\scale*5) {};

\draw (8.5cm+\scale*0,\scale* -1.2) -- (8.5cm+\scale*0, \scale *5.4);
\draw (8.5cm+\scale*3,\scale* -1.2) -- (8.5cm+\scale*3, \scale *5.4);
\draw (8.5cm+\scale*-1.2,\scale* -0.4) -- (8.5cm+\scale*9.3, \scale *3.1);

\node[draw,circle,inner sep=2pt,fill] at (8.5cm+\scale*0,\scale*-2) {};
\node[] at (8.5cm+\scale*2.2,\scale*-2) {$deg(S_A/ \langle \p_2 \rangle)$};
\draw (8.5cm+\scale*4.6,\scale* -2) -- (8.5cm+\scale*5.2, \scale *-2);
\node[] at (8.5cm+\scale*7.2,\scale*-2) {$qdeg(S_A/ \langle \p_2 \rangle)$};
\end{tikzpicture}
\end{center}

\begin{center}
\scale=0.7cm
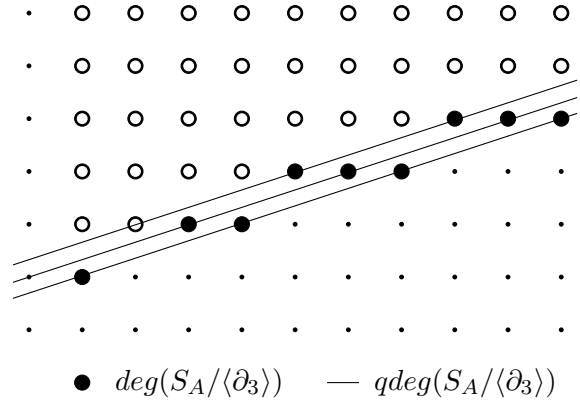
\begin{tikzpicture}
\foreach \x in {-1,0,...,9}{
   \foreach \y in {-1,0,...,5}{ 
     \node[draw,circle,inner sep=0.5pt,fill] at (\scale*\x,\scale*\y) {};
   }
 }
\node[draw,circle,inner sep=2pt,fill] at (\scale*0,\scale*0) {};
\foreach \x in {0,...,3}{
  \foreach \y in {1,...,5}{
\node[draw,circle,inner sep=2pt,fill] at (\scale*\x,\scale*\y) {};
  }
}
\foreach \x in {4,...,6}{
  \foreach \y in {2,...,5}{
\node[draw,circle,inner sep=2pt,fill] at (\scale*\x,\scale*\y) {};
  }
}
\foreach \x in {7,...,9}{
  \foreach \y in {3,...,5}{
\node[draw,circle,inner sep=2pt,fill] at (\scale*\x,\scale*\y) {};
  }
}

\node[draw,circle,inner sep=1.3pt,fill,white] at (\scale*0,\scale*1) {};
\node[draw,circle,inner sep=1.3pt,fill,white] at (\scale*1,\scale*1) {};
\foreach \x in {0,...,3}{
  \foreach \y in {2,...,5}{
\node[draw,circle,inner sep=1.3pt,fill,white] at (\scale*\x,\scale*\y) {};
  }
}
\foreach \x in {4,...,6}{
  \foreach \y in {3,...,5}{
\node[draw,circle,inner sep=1.3pt,fill,white] at (\scale*\x,\scale*\y) {};
  }
}
\foreach \x in {7,...,9}{
  \foreach \y in {4,...,5}{
\node[draw,circle,inner sep=1.3pt,fill,white] at (\scale*\x,\scale*\y) {};
  }
}

\draw (\scale*-1.3,\scale* 0.23) -- (\scale*9.3, \scale *3.73);
\draw (\scale*-1.3,\scale* -0.1) -- (\scale*9.3, \scale *3.4);
\draw (\scale*-1.3,\scale* -0.403) -- (\scale*9.3, \scale *3.1);

\node[draw,circle,inner sep=2pt,fill] at (\scale*0,\scale*-2) {};
\node[] at (\scale*2.2,\scale*-2) {$deg(S_A/ \langle \p_3 \rangle)$};
\draw (\scale*4.6,\scale* -2) -- (\scale*5.2, \scale *-2);
\node[] at (\scale*7.2,\scale*-2) {$qdeg(S_A/ \langle \p_3 \rangle)$};
\end{tikzpicture}
\end{center}
Notice that $qdeg(S_A / \langle \p_i \rangle)$ is a finite union of translations of the linear spans of faces, which do not contain $\underline{a}_i$. The set $sRes(A)$ which is the union of the sets $sRes_i(A) = -(\mbn +1)\underline{a}_i + qdeg(S_A/\langle \p_i \rangle)$ and a possible choice of the cone $\delta_A + \mbr_+ A$ are sketched below:
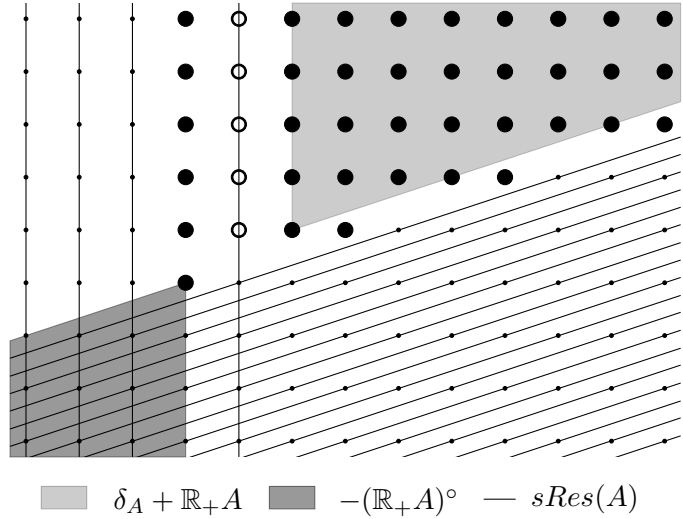
\begin{center}
\scale=0.7cm
\begin{tikzpicture}
\filldraw[black,opacity=.2] (\scale*2,\scale*1) -- (\scale*2,\scale*5.3) -- (\scale*9.3,\scale*5.3) -- (\scale*9.3,\scale*3.43) -- (\scale*2,\scale*1);
\filldraw[black,opacity=.4] (\scale*0,\scale*0) -- (\scale*0,\scale*-3.3) -- (\scale*-3.3,\scale*-3.3) -- (\scale*-3.3,\scale*-1.1) -- (\scale*0,\scale*0); 
\foreach \x in {-3,...,9}{
   \foreach \y in {-3,...,5}{ 
     \node[draw,circle,inner sep=0.5pt,fill] at (\scale*\x,\scale*\y) {};
   }
 }
\foreach \y in {3,...,5}{
  \foreach \x in {0,...,9}{
  \node[draw,circle,inner sep=2pt,fill] at (\scale*\x,\scale*\y) {};
  }
}
\foreach \x in {0,...,6}{
\node[draw,circle,inner sep=2pt,fill] at (\scale*\x,\scale*2) {};
}
\node[draw,circle,inner sep=2pt,fill] at (\scale*0,\scale*0) {};
\node[draw,circle,inner sep=2pt,fill] at (\scale*0,\scale*1) {};
\node[draw,circle,inner sep=2pt,fill] at (\scale*1,\scale*1) {};
\node[draw,circle,inner sep=2pt,fill] at (\scale*3,\scale*1) {};
\foreach \y in {1,...,5}{
  \node[draw,circle,inner sep=1.3pt,fill,white] at (\scale*1,\scale*\y) {};
}
\node[draw,circle,inner sep=2pt,fill] at (\scale*2,\scale*1) {};
\filldraw[black,opacity=.2] (-1.9,-2.7) -- (-1.3,-2.7) -- (-1.3,-3.0) -- (-1.9,-3.0) -- (-1.9,-2.7); 
\node[] at (-0.1,-2.91) {$\delta_A + \mbr_+ A$};
\filldraw[black,opacity=.4] (1.1,-2.7) -- (1.7,-2.7) -- (1.7,-3.0) -- (1.1,-3.0) -- (1.1,-2.7); 
\node[] at (2.8,-2.91) {$-(\mbr_+ A)^\circ$};
\draw (\scale*5.6,\scale* -4.1) -- (\scale*6.2, \scale *-4.1);
\node[] at (\scale*7.5,\scale*-4.1) {$sRes(A)$};
\clip (-3.3,-2.3) rectangle (\scale*9.3, \scale*5.3);
\draw (\scale*-3,\scale* -3.3) -- (\scale*-3, \scale *5.3);
\draw (\scale*-2,\scale* -3.3) -- (\scale*-2, \scale *5.3);
\draw (\scale*-1,\scale* -3.3) -- (\scale*-1, \scale *5.3);
\draw (\scale*1,\scale* -3.3) -- (\scale*1, \scale *5.3);
\draw (\scale*-3.3,\scale* -1.43) -- (\scale*9.3, \scale *2.76);
\draw (\scale*-3.3,\scale* -1.77) -- (\scale*9.3, \scale *2.43);
\draw (\scale*-3.3,\scale* -2.1) -- (\scale*9.3, \scale *2.1);
\draw (\scale*-3.3,\scale* -2.43) -- (\scale*9.3, \scale *1.76);
\draw (\scale*-3.3,\scale* -2.77) -- (\scale*9.3, \scale *1.43);
\draw (\scale*-3.3,\scale* -3.1) -- (\scale*9.3, \scale *1.1);
\draw (\scale*-3.3,\scale* -3.43) -- (\scale*9.3, \scale *0.76);
\draw (\scale*-3.3,\scale* -3.77) -- (\scale*9.3, \scale *0.43);
\draw (\scale*-3.3,\scale* -4.1) -- (\scale*9.3, \scale *0.1);
\draw (\scale*-3.3,\scale* -4.43) -- (\scale*9.3, \scale *-0.24);
\draw (\scale*-3.3,\scale* -4.77) -- (\scale*9.3, \scale *-0.56);
\draw (\scale*-3.3,\scale* -5.1) -- (\scale*9.3, \scale *-0.9);
\draw (\scale*-3.3,\scale* -5.43) -- (\scale*9.3, \scale *-1.24);
\draw (\scale*-3.3,\scale* -5.77) -- (\scale*9.3, \scale *-1.56);
\draw (\scale*-3.3,\scale* -6.1) -- (\scale*9.3, \scale *-1.9);
\draw (\scale*-3.3,\scale* -6.43) -- (\scale*9.3, \scale *-2.24);
\draw (\scale*-3.3,\scale* -6.77) -- (\scale*9.3, \scale *-2.56);
\draw (\scale*-3.3,\scale* -7.1) -- (\scale*9.3, \scale *-2.9);

\end{tikzpicture}
\end{center}
The open cone $-(\mbr_+ A)^\circ$ will become important if one considers the holonomic dual of a GKZ-system with $\beta \notin sRes(A)$ (cf. Proposition \ref{prop:dualGKZ}).
\end{ex}
\noindent The lemma above can be improved in an important special case. We call the semigroup $\mbn A$ saturated if
\begin{equation}\label{eq:defsaturated}
\mbn A = \mbq_+ A \cap \mbz^d\, .
\end{equation}
and homogeneous if there exists a linear function $h: \mbz^d \ra \mbz$ satisfying $h(\underline{a}_i)=1$ for all columns $\underline{a}_i$ of $A$.
 
\begin{lem}\label{lem:sResnormal}
Let $\mbn A$ be a saturated semigroup then
\[
\mbr_+ A \cap sRes(A) = \emptyset\, .
\]
If moreover the matrix $A$ is homogeneous, then
\[
\mbz^d \setminus sRes(A) = \mbn A\, .
\]
\end{lem}
\begin{proof}
First notice that because $\mbn A$ is saturated the true degrees of $S_A$ form exactly the set $\mbn A$. Now observe that a monomial $P$ in $S_A$ is  non-zero in $S_A / \langle \p_{\lambda_j} \rangle$ if and only if $deg P - a_j$ is not in $\mbn A$.

\noindent As we observed above $qdeg(S_F)$ is the $\mbc$-linear span of a face $F$ of $\mbn A$. If $a_j \notin F$ consider the finite set 
$$
I_{F,j} := \{ t \in [0,1) \mid (qdeg(S_F) + t \cdot a_j) \cap (\mbn A) \neq \emptyset\}\, .
$$ 
The set
$$
V_j := \bigcup_{F : a_j \notin F} \bigcup_{t \in I_{F,j}} qdeg(S_F) + t \cdot a_j = \bigcup_{F : a_j \notin F} \bigcup_{t \in I_{F,j}} \mbc F + t \cdot a_j
$$
is Zariski-closed and contains the degrees of $S_A/ \langle \p_{\lambda_j} \rangle$ and therefore $qdeg(S_A/\langle \p_{\lambda_j} \rangle) \subset V_j$. Looking now at the definition of $sRes_j(A)$ we see that
\[
sRes_j(A) = \{\beta \mid \beta \in -(\mbn+1)a_j + qdeg(S_A / \langle \p_{\lambda_j} \rangle)\}\; \subset\; \bigcup_{l \in \mbn} -(l+1)a_j + V_j \, .
\]
But we clearly have $\mbr_+ A\; \cap\;  \bigcup_{l \in \mbn}\left( -(l+1)a_j + V_j \right)= \emptyset$ and therefore $\mbr_+ A \cap sRes_j(A) = \emptyset$. This shows the first claim. 

\noindent Now assume that $A$ is homogeneous. First notice that the first claim shows $\mbn A \subset \mbz^d \setminus sRes(A)$. In order to show the other direction, let $\beta \in \mbz^d \setminus sRes(A)$ and $\beta' \in \mbn A$. We have the following isomorphisms
\begin{equation}\label{eq:betansRes}
\mcm^\beta_A \simeq \FL(h_{A,+} \mco_T) \simeq \mcm^{\beta'}_A\, .
\end{equation}
We need the following definitions of \cite{SaitoMut1}. Let $\sigma$ be a facet of $\mbq_+ A$, i.e. a codimenion one face. To each facet
we associate a unique primitive, integral support function $F_\sigma$ satisfying the following properties
\begin{enumerate}
\item $F_\sigma(\mbz A) = \mbz$,
\item $F_\sigma(a_j) \geq 0$ for all $j= 1, \ldots ,n$,
\item $F_\sigma(a_j) = 0$ for all $a_j \in \sigma$.
\end{enumerate}
Because of the isomorphisms \eqref{eq:betansRes} and \cite[Theorem 5.2]{SaitoMut1} we have
\[
F_\sigma(\beta) = F_\sigma(\beta') \geq 0
\]
which shows $\beta \in \mbq_+ A \cap \mbz^{d} = \mbn A$.
\end{proof}

%
In the next proposition we want to show that the holonomic dual of a GKZ-system $\mcm^\beta_A$ with $\beta \notin sRes(A)$ is isomorphic to $\mcm_A^{-\beta'}$ with $-\beta'$ chosen appropriately. For this we have to introduce the following set:
\[
\mbd sRes(A) := \bigcup_{F \; \text{face of}\; A} (\mbz^d \cap \mbq_+ A) + \mbc F 
\]
Notice that a parameter $\beta \in \mbc^d$ with $\beta \notin \mbd sRes(A)$ is called semi-nonresonant in \cite{SaitoMut1}. It is easy to see that 
\[
\mbd sRes(A) \cap -(\mbr_+ A)^\circ = \emptyset\, . 
\]
\begin{ex}[II continued]
Now we sketch a real picture of the set $\mbd sRes(A)$:
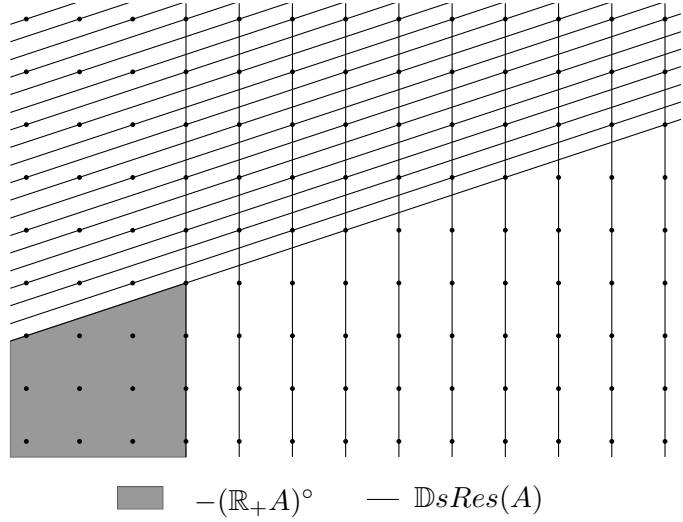
\begin{center}
\scale=0.7cm
\begin{tikzpicture}
\filldraw[black,opacity=.4] (\scale*0,\scale*0) -- (\scale*0,\scale*-3.3) -- (\scale*-3.3,\scale*-3.3) -- (\scale*-3.3,\scale*-1.1) -- (\scale*0,\scale*0); 
\foreach \x in {-3,...,9}{
   \foreach \y in {-3,...,5}{ 
     \node[draw,circle,inner sep=0.5pt,fill] at (\scale*\x,\scale*\y) {};
   }
 }
\filldraw[black,opacity=.4] (-0.9,-2.7) -- (-0.3,-2.7) -- (-0.3,-3.0) -- (-0.9,-3.0) -- (-0.9,-2.7); 
\node[] at (0.9,-2.91) {$-(\mbr_+ A)^\circ$};
\draw (\scale*3.4,\scale* -4.1) -- (\scale*4.0, \scale *-4.1);
\node[] at (\scale*5.5,\scale*-4.1) {$\mbd sRes(A)$};
\clip (-3.3,-2.3) rectangle (\scale*9.3, \scale*5.3);
\draw (\scale*-3.3,\scale* -0.43) -- (\scale*9.3, \scale *3.76);
\draw (\scale*-3.3,\scale* -0.77) -- (\scale*9.3, \scale *3.43);
\draw (\scale*-3.3,\scale* -1.1) -- (\scale*9.3, \scale *3.1);
\draw (\scale*-3.3,\scale* 0.57) -- (\scale*9.3, \scale *4.76);
\draw (\scale*-3.3,\scale* 0.23) -- (\scale*9.3, \scale *4.43);
\draw (\scale*-3.3,\scale* -0.1) -- (\scale*9.3, \scale *4.1);
\draw (\scale*-3.3,\scale* 1.57) -- (\scale*9.3, \scale *5.76);
\draw (\scale*-3.3,\scale* 1.23) -- (\scale*9.3, \scale *5.43);
\draw (\scale*-3.3,\scale* 0.9) -- (\scale*9.3, \scale *5.1);
\draw (\scale*-3.3,\scale* 2.57) -- (\scale*9.3, \scale *6.76);
\draw (\scale*-3.3,\scale* 2.23) -- (\scale*9.3, \scale *6.43);
\draw (\scale*-3.3,\scale* 1.9) -- (\scale*9.3, \scale *6.1);
\draw (\scale*-3.3,\scale* 3.57) -- (\scale*9.3, \scale *7.76);
\draw (\scale*-3.3,\scale* 3.23) -- (\scale*9.3, \scale *7.43);
\draw (\scale*-3.3,\scale* 2.9) -- (\scale*9.3, \scale *7.1);
\draw (\scale*-3.3,\scale* 4.57) -- (\scale*9.3, \scale *8.76);
\draw (\scale*-3.3,\scale* 4.23) -- (\scale*9.3, \scale *8.43);
\draw (\scale*-3.3,\scale* 3.9) -- (\scale*9.3, \scale *8.1);
\draw (\scale*-3.3,\scale* 4.9) -- (\scale*9.3, \scale *9.1);
\draw (\scale*0,\scale* -3.3) -- (\scale*0, \scale *5.3);
\draw (\scale*1,\scale* -3.3) -- (\scale*1, \scale *5.3);
\draw (\scale*2,\scale* -3.3) -- (\scale*2, \scale *5.3);
\draw (\scale*3,\scale* -3.3) -- (\scale*3, \scale *5.3);
\draw (\scale*4,\scale* -3.3) -- (\scale*4, \scale *5.3);
\draw (\scale*5,\scale* -3.3) -- (\scale*5, \scale *5.3);
\draw (\scale*6,\scale* -3.3) -- (\scale*6, \scale *5.3);
\draw (\scale*7,\scale* -3.3) -- (\scale*7, \scale *5.3);
\draw (\scale*8,\scale* -3.3) -- (\scale*8, \scale *5.3);
\draw (\scale*9,\scale* -3.3) -- (\scale*9, \scale *5.3);
\end{tikzpicture}
\end{center}
\end{ex}
\begin{prop}\label{prop:dualGKZ}
Let $A$ be homogeneous and $\beta \in  \mbq^{d} \setminus sRes(A)$. Then
\[
\mbd \mcm^{\beta}_{A} \simeq \mcm^{\beta'}_{A}
\]
for all $\beta' \notin \mbd sRes(A)$ with $\beta \equiv -\beta'\; \textup{mod}\; \mbz^d$.
\end{prop}
\begin{proof}
By \cite[Theorem 4.8] {Walther1} there exists a Zariski open subset $U \subset \mbc^{d}$ of parameters (depending on $A$) such that $\mbd \mcm^{\beta}_{A}$ is isomorphic to $\mcm^{-\beta - \epsilon_{A'}}_{A}$ for all $\beta \in U$, where
$\epsilon_{A'}$ is the sum of the columns of a matrix $A'$, where $A'$ is a matrix which generates the saturation of $\mbn A$.

\noindent Now let $\beta \in \mbq^d \setminus sRes(A)$, there exists an $\alpha \in \mbz^{d}$ such that $\beta + \alpha \in U \cap (\delta_{A} + \mbr_+ A)$. We have the following isomorphisms
\[
\mbd \mcm^{\beta}_{A} \simeq \mbd \mcm^{\beta + \alpha}_A \simeq \mcm^{- \beta - \alpha - \epsilon_{A}'}_{A} \simeq \mcm^{\beta'}_{A}\, .
\]
The first isomorphism holds because of Theorem \ref{thm:SW} and Lemma \ref{lem:sResnonnormal}  and the last isomorphism holds because $- \beta - \alpha - \epsilon_{A'} \in -(\mbr_+ A)^{\circ}$ and \cite{SaitoMut1} Corollary 2.6.
\end{proof}

In order to prove the dual statement of the theorem of Schulze and Walther above, we will need the following lemma.

\begin{lem}\label{lem:GKZmono}
Let $A$ be homogeneous and $\beta \in \mbc^d$, then the GKZ-system $\mcm^\beta_A$ is monodromic.
In particular, the $\mcd$-module $h_{A,+}(\mco_{T'} \underline{y}^\beta)$ is monodromic for every $\beta \notin sRes(A)$.
\end{lem}
\begin{proof}
If $A$ is homogeneous, by definition there exists $h: \mbz^d \ra \mbz$ with $h(\underline{a}_i) =1$. This gives rise to a linear combination of the Euler field  $\mathfrak{E}$ in terms of the operators $E_k$:
\[
\mathfrak{E} = \sum_{i=1}^n \lambda_i \p_i = \sum_{k=1}^d h_k E_k,
\]
where $h_k$ are the components of $h$ with respect to the standard basis of $\mbz^d$. Therefore $\mathfrak{E} - b$ is zero in $\mcm_A^\beta$, where $b := \sum_{k=1}^d h_k \beta_k$. This shows  that $\mathfrak{E}-b$ acts on the monomial
\[
\prod_{i=1}^n \lambda_i^{d_i} \cdot \prod_{i=1}^n \p_i^{e_i}
\]
by multiplication with $\sum_{i=1}^n d_i - \sum_{i=1}^n e_i$. Taking a general element $v$, the space $\mathfrak{E}^{l}(v)$ is a subspace of the finite-dimensional $\mbc$-vectorspace generated by the monomials appearing in $v$, which shows the first claim.

The second claim follows from the fact that $h_{A,+} \mco_T\underline{y}^{\beta} \simeq \FL(\mcm^\beta_A)$ for $\beta \notin sRes(A)$ and the fact that the Fourier-Laplace-transform sends monodromic $\mcd$-modules to monodromic $\mcd$-modules.
\end{proof}

\begin{prop}\label{prop:dualSW}
Let $A$ be a homogeneous $d \times n$ integer matrix and let $\beta' \in \mbq^d$ with $\beta' \notin \mbd sRes(A)$, then
\[
\FL(h_{A,\dag} \mco_{T'} \underline{y}^{\beta'}) \simeq \mcm_A^{\beta'}\, .
\]
\end{prop}
\begin{proof}
First notice that we have
\[
\FL(h_{A,\dag} \mco_{T'}\underline{y}^{\beta'}) \simeq \FL(h_{A,\dag} \mco_{T'}\underline{y}^{\beta''})
\]
for $\beta'' \in - (\mbr_+ A)^\circ$  with $\beta'' \equiv \beta'\; mod\; \mbz^d$. By the definition of $h_{A,\dag}$ we have
\[
\FL(h_{A,\dag} \mco_{T'}\underline{y}^{\beta''}) \simeq  \FL( \mbd h_{A,+}(\mco_{T'} \underline{y}^{\delta_A - \beta''}))\, ,
\]
where we have used that $\mbd \mco_{T'}\underline{y}^{\beta''} \simeq \mco_{T'} \underline{y}^{-\beta'' + \alpha}$ holds for every $\alpha \in \mbz$.
Because of  Lemma \ref{lem:GKZmono} and the fact that the monodromic Fourier-Laplace transform commutes with duality, we have the following isomorphism
\[
\FL( \mbd h_{A,+}(\mco_{T'} \underline{y}^{\delta_A - \beta''})) \simeq \mbd \FL( h_{A,+}(\mco_{T'} \underline{y}^{\delta_A - \beta''})).
\]
We conclude by applying Theorem \ref{thm:SW} and Lemma \ref{lem:sResnonnormal}
\[
\mbd \FL(h_{A,+}(\mco_{T'} \underline{y}^{\delta_A - \beta''})) \simeq \mbd \mcm_A^{\delta_A - \beta''} \simeq \mcm_A^{\beta'}\, ,
\]
where the last isomorphism follows from Proposition \ref{prop:dualGKZ}.
\end{proof}

We need still another fact of the theory of hypergeometric systems $\mcm_{A}^{\beta}$. 
Let $\beta, \beta' \in \mbc^{d}$ so that $\beta - \beta' \in \mbn A$.
We denote by $\p^{\beta - \beta'}$ the element $\prod^n_{i=1} \p_i^{\gamma_i}$ in $\mcm^\beta_A$, where $(\gamma_i)_{i=1,\ldots,n}$ is chosen such that $(\beta- \beta')_k = \sum_{i=1}^n a_{ki} \gamma_i$. Notice that this is well-defined in $\mcm^\beta_A$ due to the box-operators $(\Box_l)_{l \in \mbl}$.  There is the following rigidity result for morphisms between such GKZ-systems.
\begin{prop}\label{prop:gkzuniquemorph}
Let $\psi: \mcm^{\beta'}_{A} \lra \mcm^{\beta}_{A}$ be a non-zero $\mcd$-linear morphism, then (up to multiplication with a non-zero constant) $\psi$ is given by right multiplication with $\p^{\beta - \beta'}$.
\end{prop} 
\begin{proof}
In the course of the proof we will work with the modules of global sections instead of the $\mcd$-modules themselves.
First notice that $\psi$  is determined by the image of $[1] \in M^{\beta'}_{A}$ because $M^{\beta'}_{A}$ is a cyclic left $D$-module. Let $P \in D$ so that $\psi([1]) = [P]$. Define for $\kappa \in \mbc^{d}$ the $\mbc[\lambda_1 \p_{\lambda_1}, \ldots , \lambda_n \p_{\lambda_n}]$-module
\[
(M^{\beta}_{A})_{\kappa} :=  \{ [Q] \in M^{\beta}_{A} \mid (E_k + \kappa_k) \cdot [Q] = 0 \; \text{for all} \; k \in \{1 , \ldots , d\} \}\,.
\]
From \cite[Chapter 4]{SaitoMut3} follows that there exists the following $\mbc[\lambda_1 \p_{\lambda_1}, \ldots , \lambda_n \p_{\lambda_n}]$-module isomorphism for $\beta + \kappa \in \mbn A$:
\begin{align}
\mbc[\lambda_1 \p_{\lambda_1}, \ldots , \lambda_n \p_{\lambda_n}]/ \sum_{k=1}^d (E_k + \kappa_k) \mbc[\lambda_1 \p_{\lambda_1}, \ldots , \lambda_n \p_{\lambda_n}] &\lra (M^{\beta}_{A})_{\kappa} \label{eq:gkzgradiso}\\
g(\lambda \p_{\lambda}) := g(\lambda_1 \p_{\lambda_1}, \ldots  , \lambda_n \p_{\lambda_n}) &\mapsto g(\lambda \p_{\lambda})\p^{\beta + \kappa} \notag
\end{align}

In order that $\psi$ is well-defined, we must have $[P] \in (M^{\beta}_{A})_{- \beta'}$. Therefore $P$ can be chosen to lie in $\mbc[\lambda_1 \p_{\lambda_1}, \ldots , \lambda_n \p_{\lambda_n}] \p^{\beta  - \beta'}$. We will write  $\psi([1]) = [f(\lambda \p_{\lambda})\p^{\beta - \beta'}]$ for short and compute the image of $\Box_{\underline{l}} = \p^{l^{-}} - \p^{l^+}$
\begin{align}
0 &= \psi([\p^{l^{-}} - \p^{l^+}])\notag \\
&= (\p^{l^{-}} - \p^{l^+}) \cdot \psi([1]) = (\p^{l^{-}} - \p^{l^+}) [f(\lambda \p_{\lambda})\p^{\beta - \beta'}] \notag \\
&= [(f(\lambda \p_{\lambda} + l^{-}) \p^{l^{-}} - f(\lambda \p_{\lambda} + l^{+})\p^{l^{+}})\p^{\beta - \beta'}] \notag \\
&= [(f(\lambda \p_{\lambda} + l^{-})  - f(\lambda \p_{\lambda} + l^{+}))\p^{l^+} \p^{\beta - \beta'}] \notag \\
& \in (M^{\beta}_{A})_{Al^+ - \beta'}\, ,
\end{align}
where $A l^+ \in \mbz^{d}$ with $k$-th component $(A l^+)_k := \sum_{i=1}^n a_{ki}(l^+)_i$. Because of \eqref{eq:gkzgradiso} we have
\[
f(\lambda \p_{\lambda} + l^{-})  - f(\lambda \p_{\lambda} + l^{+}) \in \sum_{k=1}^d (E_k + (A l^+)_k - \beta'_k) \mbc[\lambda_1 \p_{\lambda_1}, \ldots , \lambda_n \p_{\lambda_n}] \, .
\]

\noindent Notice that for a relation $l \in \mbn^{n}$ we have
\[
f(\lambda \p_{\lambda})  - f(\lambda \p_{\lambda} + l) \in \sum_{k=1}^d (E_k - \beta'_k) \mbc[\lambda_1 \p_{\lambda_1}, \ldots , \lambda_n \p_{\lambda_n}]\, .
\]
The statement above is a statement in a commutative ring. For better readability we set $x_i = \lambda_i \p_{\lambda_i}$, then the statement above can be expressed as (recall that $E_k = \sum_{i=1}^n a_{ki} \lambda_i \p_{\lambda_i}$)
\[
f(x) - f(x+l) \in \sum_{k=1}^d((A\cdot x)_k -\beta'_k)\mbc[x]\, .
\] 
Because the columns of $A$ generate $\mbz^{d}$, i.e $A$ has full rank, we can find a $\gamma \in \mbc^{d}$ with $A \cdot \gamma = \beta'$. Thus we have
\[
f(\gamma) - f(\gamma +l) = 0 \, .
\]
Since $f$ is a polynomial this means that $f$ has constant value $f(\gamma)$ one the affine subspace $\gamma + \ker(A)$. But this means
\[
f(\gamma +x) \in f(\gamma) +\sum_{k=1}^d(A \cdot x)_k \mbc[x]\, ,
\]
resp.
\[
f(x) \in f(\gamma) +\sum_{k=1}^d((A \cdot x)_k -\beta'_k) \mbc[x]\, .
\]
If we substitute $\lambda \p_{\lambda}$ back, we get
\[
\psi([1]) =[P] =[f(\lambda \p_{\lambda})\p^{\beta - \beta'}] = [f(\gamma)\p^{\beta - \beta'}]\, ,
\] 
where we have used $(E_k- \beta'_k)\p^{\beta - \beta'} =\p^{\beta - \beta'}(E_k- \beta_k)$. This shows the claim.
\end{proof}

\noindent In the rest of this section we will consider special types of GKZ-system.
Let $A$ be a $d \times n$-matrix with $\mbz A = \mbz^d$. We define its homogenization as being the  $d+1 \times n+1$-matrix
\begin{equation}\label{eq:tildeA}
\widetilde{A} := \left(\begin{matrix}1 & 1 & \ldots & 1\\ 0 &&& \\ \vdots & & \text{{\Huge $A$}} & \\ 0 &&& \end{matrix} \right)\, .
\end{equation}
We will consider the GKZ-system $\mcm^{\widetilde{\beta}}_{\widetilde{A}}$ with $\widetilde{\beta} = (\beta_0, \ldots, \beta_d) \in \mbc^{d+1}$ and denote the coordinates of the underlying space as $\lambda_0 ,\ldots, \lambda_m$.

\noindent Notice that the semigroup $\mbn \widetilde{A}$ is always pointed, thus every statement above applies to these kind of matrices.\\

\noindent Later we will need the following lemma.
\begin{lem}\label{lem:betalift}
Let $A$ be a $d \times n$ integer matrix with $\mbz A = \mbz^d$ and let $\beta \in \mbq$. If $\beta \notin sRes(A)$, there exist a $n_\beta \in \mbz$ such that $\widetilde{\beta} = (\beta_0, \beta) \notin sRes(\widetilde{A})$ for all $\beta_0 \in \mbq$ with $\beta_0 \geq n_\beta$.
\end{lem}
\begin{proof}
Fix a $\beta \in \mbq^d$. At first we prove that there exists an $n_\beta \in \mbz$, s.t. for $\beta_0 \geq n_\beta$ the element $(\beta_0, \beta) \notin sRes_0(\widetilde{A})$. For this we have to compute the quasi-degrees $\mbq^{d+1} \cap qdeg(S_{\widetilde{A}}/\p_{\lambda_0})$.
Recall that $\mbq^{d+1} \cap qdeg(S_{\widetilde{A}}/\p_{\lambda_0})$ is a finite union of translates of $\mbq$-linear spans of faces of $\mbq_+ \widetilde{A}$ which do not contain $\underline{\widetilde{a}}_0 = (1,0, \ldots ,0)$ (cf.\eqref{eq:quasideg}). 
Thus we can find an $n_\beta \in \mbz$ so that for every $\beta_0 \geq n_\beta$ the element $(\beta_0, \beta) \notin sRes_0(\widetilde{A}) = -(\mbn +1)\widetilde{\underline{a}}_0 + qdeg(S_{\widetilde{A}}/\p_0)$.

\noindent Now assume additionally that $\beta \notin sRes_j(A)$ for some $j \in \{1, \ldots ,n\}$. Recall that this means
\[
\beta \notin -(\mbn +1) \underline{a}_j + qdeg(S_A / \p_{\lambda_j}).
\] 
Notice that $qdeg(S_{\widetilde{A}}/ \p_{\lambda_j}) \subset \mbq \times qdeg(S_A / \p_{\lambda_j})$. But this means
\begin{align}
sRes_j(\widetilde{A}) &= -(\mbn +1) \underline{\widetilde{a}}_j + qdeg(S_{\widetilde{A}}/\p_{\lambda_j}) \notag \\
&\subset -(\mbn +1)(\mbq \times \underline{a}_j) + (\mbq \times qdeg(S_A / \p_{\lambda_j})) \notag \\
&= \mbq \times sRes_j(A)\, . \notag
\end{align}
But this means that $(\beta_0, \beta) \notin sRes_j(\widetilde{A})$ for any $\beta_0 \in \mbq$. Summarizing we have shown that if $\beta \notin sRes(A)$, then $(\beta_0, \beta) \notin sRes(\widetilde{A})$ for any $\beta_0 \in \mbq$ with $\beta_0 \geq n_\beta$, but this shows the claim.
\end{proof}

\section{Families of Laurent Polynomials and Hypergeometric $\mcd$-modules}\label{sec:LaurentPolGKZ}

In this section we will show that certain (direct sums of) homogeneous GKZ-systems arise as Radon transformations of some $\mcd$-modules on the projective space. This approach to  GKZ-systems is a crucial step in order to endow them with a mixed Hodge module structure. Additionally, this enables us to establish a strong relationship between (these direct sum of) GKZ-systems and the Gau\ss-Manin system of an associated family of Laurent polynomials (Theorem \ref{thm:genseq}), which will give interesting insight into the structure of these Gau\ss-Manin systems.

\noindent Let $B$ be an integer $d \times n$-matrix, where we denote the columns by $\underline{b}_1, \ldots , \underline{b}_n$. We assume that the columns of $B$ generate $\mbq^d$.  

\noindent Using the Smith normal form we can write $B$ as
\[
C \cdot D_1 \cdot D_2 \cdot M
\]
with $C=(\underline{c}_1, \ldots , \underline{c}_{d}) \in Gl(d \times d , \mbz), M=(\underline{m}_1, \ldots , \underline{m}_n) \in GL(n \times n, \mbz)$ and
\[
D_1 \cdot D_2= \left( \begin{array}{c  c  c} e_1 & &  \\ & \ddots & \\ & & e_{d} \end{array} \right) \cdot \left( \begin{array}{c  c  c| c c c} 1 & & & & &\\ & \ddots & &  & 0 & \\ & &1 & & &\end{array} \right)\, .
\]

\noindent We then set
\[
A:= D_2 \cdot M
\]
and we consider its homogenization $\widetilde{A}$ as in  \eqref{eq:tildeA}.

\noindent Furthermore, we associate to the matrix $B$ the following family of Laurent polynomials
\begin{align}
\varphi_{B}: S \times W &\lra \mbc_{\lambda_0} \times W\, , \notag \\
(y_1, \ldots , y_{d}, \lambda_1, \ldots , \lambda_n) &\mapsto (-\sum_{i=1}^n \lambda_i \underline{y}^{\underline{b}_i}, \lambda_1, \ldots , \lambda_n)\, , \notag
\end{align}
where $S := (\mbc^*)^{d}$, $W := \mbc^n$.

\noindent Set $V:= \mbc_{\lambda_0} \times W$.  We will construct a $\mcd_V$-linear morphism with $\mco_V$-free kernel and cokernel between the Gau\ss-Manin system $\mch^0(\varphi_{B+} \mco_{S \times W})$ and a direct sum of GKZ-systems. 

\begin{thm}\label{thm:genseq}
Let $B$ and $\widetilde{A}$ be as above and let $\varphi_{B}: S \times W \lra \mbc_{\lambda_0} \times W$ be the corresponding family Laurent polynomials. Let 
\begin{align}
\sigma : (\mbq / \mbz)^{d+1} &\ra \mbq^{d+1}\setminus sRes(\widetilde{A}) \notag \\
\sigma': (\mbq / \mbz)^{d+1} &\ra \mbq^{d+1} \setminus \mbd sRes(\widetilde{A}) \notag
\end{align}
be sections of the projection $pr: \mbq^{d+1} \ra (\mbq/\mbz)^{d+1}$ and let $I := (\mbz \times \frac{1}{e}\mbz^d) \cap im(\sigma)$ resp. $I' := \mbz \times \frac{1}{e}\mbz^d \cap im(\sigma')$, where $\frac{1}{e}\mbz^d = \frac{1}{e_1}\mbz \times \ldots \times \frac{1}{e_d}\mbz$ and $e =(e_1, \ldots e_d)$ are the elementary divisors of $B$. Then we have the following exact sequences in $M_{rh}(\mcd_V)$:
\begin{align}
&0 \lra H^{d-1}(S,\mbc)\otimes \mco_V \lra \mch^{0}(\varphi_{B,+} \mco_{S \times W}) \lra \bigoplus_{\widetilde{\beta} \in I} \mcm^{\widetilde{\beta}}_{\widetilde{A}} \lra H^{d}(S,\mbc)\otimes \mco_V \lra 0\, , \notag \\
&0 \lra H_{d}(S,\mbc)\otimes \mco_V \lra \bigoplus_{\widetilde{\beta}' \in I'} \mcm^{\widetilde{\beta}'}_{\widetilde{A}}  \lra \mch^{0}(\varphi_{B,\dag} \mco_{S \times W}) \lra H_{d-1}(S,\mbc)\otimes \mco_V \lra 0\, . \notag
\end{align}
\end{thm}

\begin{rem}
A priori the existence of a section $\sigma : (\mbq / \mbz)^{d+1} \ra \mbq^{d+1}\setminus sRes(\widetilde{A})$ is not clear, however using Lemma \ref{lem:sResnonnormal} we can guarantee the existence of such a section with image contained in $\delta'_{\widetilde{A}} + (\{0\} \times [0,1)^{d})$ for some $\delta'_{\widetilde{A}} \in \delta_{\widetilde{A}} + \mbq_+ \widetilde{A}$.
\end{rem}

\noindent In the case where the columns of $B$ generate $\mbz^{d}$, i.e. in the case $B =A$, we can be more precise with respect to the allowed parameter vector $\beta$.

\begin{cor}\label{cor:exseqA}
Let $A$ be an integer $d \times n$-matrix with $\mbz A = \mbz^d$ and let $\varphi_A$ be the corresponding family Laurent polynomials. For every $\widetilde{\beta}, \widetilde{\beta}' \in  \mbz^{d+1}$ with $\widetilde{\beta} \notin sRes(\widetilde{A})$ resp. $\widetilde{\beta}' \notin \mbd sRes(\widetilde{A})$ we have the following exact sequences in $M_{rh}(\mcd_V)$:
\begin{equation} \label{eq:exseq1}
0 \lra H^{d-1}(S,\mbc)\otimes \mco_V \lra \mch^0(\varphi_{A,+} \mco_{S \times W}) \lra  \mcm^{\widetilde{\beta}}_{\widetilde{A}} \lra H^{d}(S,\mbc)\otimes \mco_V \lra 0\, .
\end{equation}
\begin{equation}\label{eq:exseq2}
0 \lra H_d(S,\mbc)\otimes \mco_V \lra \mcm^{\widetilde{\beta}'}_{\widetilde{A}}  \lra \mch^0(\varphi_{A,\dag} \mco_{S \times W}) \lra H_{d-1}(S,\mbc)\otimes \mco_V \lra 0\, .
\end{equation}
If in addition $\mbn \widetilde{A}$ is saturated, the set $\{\widetilde{\beta} \in \mbz^{d+1} \mid \widetilde{\beta} \notin sRes(\widetilde{A})\}$ is precisely $\mbn \widetilde{A}$.
\end{cor}
\begin{proof}
The statements follow from the fact, that the elementary divisors of the matrix $A$ are all equal to $1$. The last statement is the second statement of Lemma \ref{lem:sResnormal}.
\end{proof}

In order to derive the exact sequences in Theorem \ref{thm:genseq} we have to recall briefly the definition and some simple facts about the Radon transformation for $\mcd$-modules.

\noindent  On the level of $\mcd$-modules the Radon transform was discussed by \cite{Brylinski} and some variants were later discussed in \cite{AE}.

\noindent Consider the following diagram
$$
\xymatrix{ && U \ar[drr]^{\pi_2^U} \ar[dll]_{\pi_1^U} \ar@{^(->}[d]^{j_U}&& \\ \mbp(V') && \mbp(V') \times V \ar[ll]_{\pi_1} \ar[rr]^{\pi_2} && V\; , \\ && Z \ar[ull]^{\pi_1^Z} \ar@{^(->}[u]_{i_Z} \ar[rru]_{\pi_2^Z} &&}
$$
where we denote by $Z$ the hypersurface given by the equation $\sum_{i=0}^n \mu_i \lambda_i = 0$ and by $U$ its complement in $\mbp(V') \times V$. Consider the following functors from $D^b_{rh}(\mcd_{\mbp(V')})$ to $D^b_{rh}(\mcd_{V})$:
\begin{align}
\mcr(M)&:= \pi^Z_{2 +}\, (\pi^Z_1)^+ M \simeq \pi_{2 +}\, i_{Z +}\, i_{Z}^+\, \pi_1^+ M  \, , \notag \\
\mcr^\circ(M)&:= \pi^U_{2 +}\, (\pi^U_1)^+ M \simeq \pi_{2+}\, j_{U +} j^{+}_U \pi_1^+ M\, , \notag \\
\mcr^\circ_c(M)&:= \pi^U_{2 \dag}\, (\pi^U_1)^+ M \simeq \pi_{2 +}\, j_{U \dag}\, j^{+}_U \pi_1^+ M\, , \notag \\
\mcr_{cst}(M)&:= \pi_{2 +}\, (\pi_1)^+ M\, . \notag
\end{align}

\noindent Notice that the various Radon transformations give rise to the following triangles:
\begin{prop}\label{prop:Radontriangle}Let $M \in D^b_{rh}(\mcd_{\mbp(V')})$, we have
\begin{align}
\mcr(M)[-1] \lra \mcr_{cst}(M) \lra \mcr^\circ(M) \overset{+1}{\lra}\, , \notag \\
\mcr^\circ_{c}(M) \lra \mcr_{cst}(M) \lra \mcr(M)[1] \overset{+1}{\lra}\, , \notag
\end{align}
where the second triangle is dual to the first.
\end{prop}
\begin{proof}
The first triangle follows from the adjunction triangle 
\[
(i_Z)_+ (i_Z)^+[-1] \lra id \lra (j_U)_+ (j_U)^+ \overset{+1}{\lra}\; .
\]
The second triangle is dual to the first.
\end{proof}

\noindent The following proposition relates the Fourier and Radon transformations introduced above, and will be quite useful in the next chapter.

\begin{prop}{\cite[Proposition 1]{AE}}\label{lem:AERadonFourier}
Let $V$ be a $\mbc$-vector space, $V'$ its dual space,
$p: Bl_0(V'):=\mbv(\mco_{\mbp(V')}(-1)) \rightarrow V'$ the blowup of $V'$ at the origin, and consider the following diagram
$$
\begin{xy}
\xymatrix{ & V' \ar@{.>}[rrd]^{\FL} & & \\ Bl_0(V') \ar[ru]^{p} \ar[rd]_q & V'\setminus\{0\} \ar[u]^{j} \ar[d]_{\pi} \ar[l] &  & V\, .  \\ & \mbp(V') \ar@/^1pc/@{.>}[rru]^{\mcr^\circ_{(c)}} \ar@/_1pc/@{.>}[rru]_{\mcr_{(cst)}} &  & }
\end{xy}
$$
Let $M \in D^b_{rh}(\mcd_{\mbp(V')})$. Then we have
\begin{align}
\mcr(M) &\simeq \FL(p_+ q^+ M)\, , \notag \\
\mcr^\circ_c(M) &\simeq \FL(j_{+}\pi^+ M)\, , \notag\\
\mcr^\circ(M) &\simeq \FL(j_{\dag}\pi^+ M)\, . \notag
\end{align}
\end{prop}

\noindent Consider the following map
\begin{align}
g: S &\lra \mbp^n\, , \notag \\
(y_1, \ldots , y_d) &\mapsto (1:\underline{y}^{\underline{b}_1}: \ldots : \underline{y}^{\underline{b}_n})\, , \label{eq:embed}
\end{align}

\noindent In order to construct the morphisms between the (proper) Gau\ss-Manin system of $\varphi_B$ and the direct sum of GKZ-systems, consider the following exact triangles  in $D^b_{rh}(V)$ from Proposition \ref{prop:Radontriangle}:
\begin{align}
\mcr(g_\dag\, \mco_S)[-1] \lra \mcr_{cst}(g_\dag\, \mco_S) \lra \mcr^\circ(g_\dag\, \mco_S) \overset{+1}{\lra} &\\
\mcr^\circ_{c}(g_+\, \mco_S) \lra \mcr_{cst}(g_+\, \mco_S) \lra \mcr(g_+\, \mco_S)[1]   \overset{+1}{\lra}&\, . \notag
\end{align}
In the following we will calculate each term of the triangles above.\\

\noindent Consider the following map
\begin{align}
h: T &\lra V' \label{def:maph} \\
(y_0, \ldots , y_d) &\mapsto (y_0, y_0 \cdot \underline{y}^{\underline{b}_1}, \ldots , y_0 \cdot \underline{y}^{\underline{b}_n})\, , \notag
\end{align}
where $T \simeq (\mbc^\ast)^{d+1}$, $V' \simeq \mbc^{n+1}$.

\noindent These maps give rise to the following diagram, where the lower square is cartesian.
$$
\xymatrix{
 &  V'  \\
T \ar[ur]^{h} \ar[d]_{\pi_T}  \ar[r]_{\widetilde{h}} &  V' \ar[d]_{\pi} \ar[u]_{j} \setminus \{0\}\\
S \ar[r]^{g} &  \mbp(V')}
$$
Here, the map $\pi$ is the canonical projection and $j$ is the canonical inclusion. The map $\pi_T$ is the projection given by
\begin{align}
\pi_T : T &\lra S\; , \notag \\
(y_0,y_1,\ldots , y_d) &\mapsto (y_1, \ldots ,y_d)\; . \notag
\end{align}

\noindent In the next lemma we compare various $\mcd$-modules constructed from the $\mcd$-modules $\mco_S$ resp. $\mco_T := \pi_T^+ \mco_{S}$, living on the $d$-dimensional torus $S$ resp. $d+1$-dimensional torus $T$.
\begin{lem}\label{lem:Step1}$ $\\[-1.5 em]
\begin{enumerate}
\item
The functors $h_+$ and $h_\dag$ are exact.
\item
We have isomorphisms
\begin{align}
j_+ \pi^{+} g_+ \mco_{S} &\simeq h_+ \mco_{T}\, , \notag \\
j_\dag \pi^{+} g_\dag \mco_{S} &\simeq h_\dag \mco_{T}\, , \notag
\end{align}
in the category of monodromic, $\mcd$-modules on $V'$.
\end{enumerate}
\end{lem}

\begin{proof}$ $\\[-1 em]
\begin{enumerate}
\item  The exactness of $h_+$ follows from the fact that the map $h$ is an quasi-finite, affine and from \cite[VI,Proposition 8.1]{Borel}. The exactness of $h_\dag$ follows by duality. 

\item To prove the second point, observe that the following diagram is cartesian:
$$
\xymatrix{
T  \ar[d]_{\pi_T}  \ar[r]^{\widetilde{h}}
&  V'\setminus \{0\} \ar[d]_{\pi}   & \\
S \ar[r]^{g} &  \mbp(V')\, .& }
$$
Using base change with respect to $\pi$ (see e.g. \cite[Theorem 1.7.3]{Hotta}) we get
\[
\pi^+ g_+ \mco_{S} \simeq \widetilde{h}_+  \pi_T^+ \mco_{S} \simeq \widetilde{h}_+ \mco_T \, ,
\]
where we have used that  $\pi_T^{+} \,\mco_{S} \simeq \mco_{T}$.\\

From this follows
\[
h_+ \mco_T \simeq j_+ \widetilde{h}_+ \mco_T \simeq j_+ \pi^+ g_+ \mco_{S}\, .
\]
The second isomorphism follows by duality and the fact that $\pi^\dag = \pi^+$.
\end{enumerate}
\end{proof}

\noindent In the next proposition we compare the direct image of $\mco_{S \times W}$ under $\varphi_B$ with the Radon transform of $g_+ \mco_{S}$. Here and in the following we will identify $V$ with $\mbc_{\lambda_0} \times W$.
\newpage
\begin{prop}\label{lem:Step2}$ $\\[-1 em]
\begin{enumerate}
\item
Let $\varphi_B: S \times W \rightarrow \mbc_{\lambda_0} \times W $ be the family of Laurent polynomials defined above. Then we have the following isomorphisms in $D^b_{rh}(\mcd_V)$
\begin{align}
\mathcal{R}( g_+ \mco_{S}) &\simeq \varphi_{B,+}  \mco_{S \times W}\, , \notag \\
\mathcal{R}( g_\dag \mco_{S}) &\simeq \varphi_{B,\dag}  \mco_{S \times W}\, . \notag
\end{align}
\item There are isomorphisms
\begin{align}
\mch^{i}(\mcr_{cst}(g_+ \mco_{S})) &\simeq H^{d+i}(S, \mbc)\otimes\mco_V\, , \notag \\
\mch^{i}(\mcr_{cst}(g_\dag \mco_{S})) &\simeq H_{d-i}(S, \mbc)\otimes\mco_V\,  \notag
\end{align}
in $D^b_{rh}(\mcd_V)$.
\item
There are isomorphisms
\begin{align}
\mcr^\circ_c(g_+ \mco_{S}) &\simeq \FL(h_+ \mco_{T})\, , \notag \\
\mcr^\circ (g_\dag \mco_{S}) &\simeq \FL(h_\dag \mco_{T})  \notag
\end{align}
in $D^b_{rh}(\mcd_V)$.
\end{enumerate}
\end{prop}
\begin{proof}
\begin{enumerate}
\item
Consider the following diagram, where the square is cartesian
$$
\xymatrix{S\times W \ar[r]^{\widetilde{i}}& \Gamma \ar[d]_\eta \ar[r]^\kappa & Z \ar[d]^{\pi_1^Z} \ar[r]^{\pi_2^Z} & V \\ & S \ar[r]^{g}  & \mbp(V') &  & \qquad .}
$$
Recall that the hypersurface $Z$ in $\mbp(V') \times V$ is given by $\sum_{i=0}^n \lambda_i \mu_i =0$ and $\Gamma$ is the fibered product of $S$ and $Z$. The map $g:S \ra \mbp(V')$ is given by
\begin{align}
g: S &\lra \mbp(V')\, , \notag \\
(y_1, \ldots , y_d) &\mapsto (1: \underline{y}^{\underline{b}_1}: \ldots : \underline{y}^{\underline{b}_n})\, . \notag
\end{align}
 Thus $\Gamma$ is the smooth hypersurface in $S \times V$ given by $\lambda_0 + \sum_{i=1}^n \lambda_i \underline{y}^{\underline{b}_i}=0$. Notice that we have an isomorphism $\widetilde{i}: S \times W \lra \Gamma$ given by
\[
(y_1, \ldots , y_d,\lambda_1, \ldots , \lambda_n) \mapsto  (y_1, \ldots , y_d, -\sum_{i=1}^n \lambda_i \underline{y}^{\underline{b}_i},\lambda_1, \ldots , \lambda_n)\,.
\]
\noindent We have
\[
 \mcr(g_+ \mco_{S}) \simeq  \pi_{2 +}^Z\, (\pi_1^{Z})^+\, g_+ \,\mco_{S} \simeq \pi_{2+}^Z \kappa_+\,  \eta^{+}\,  \mco_{S}
\]
by base change with respect to the map $\pi_1^Z$.

\noindent Notice that $\pi_2^Z \circ \kappa \circ \widetilde{i} = \varphi_{B}$ by the definition of $\Gamma$. Using the fact that $\eta^{+}\, \mco_{S} \simeq \widetilde{i}_+ \mco_{S \times W}$, we obtain
\[
\mathcal{R}( g_+ \mco_{S}) \simeq \varphi_{B,+}\,  \mco_{S \times W}\, .
\]
The second statement is just the dual statement of the first.
\item Consider the cartesian diagram
\[
\begin{xy}
\xymatrix{   & \mbp(V') \times V \ar[rd]_{\pi_2} \ar[ld]^{\pi_1} & \\   \mbp(V') \ar[dr]^{a_\mbp} & & V\, . \ar[dl]_{a_V}\\ & pt &}
\end{xy}
\]
We have
\[
\mcr_{cst}( g_+ \mco_{S}) \simeq \pi_{2 +}\, \pi_1^{+} g_+ \mco_{S} \simeq a_V^+\, a_{\mbp +}\,  g_+ \mco_{S} \simeq a_V^+ a_{S+} \mco_S
\]
by base change with respect to $a_V$.
We get
\[
\mch^{i} \mcr_{cst}( g_+ \mco_{S}) \simeq H^{d+i}(S,\mbc) \otimes \mco_V\,.
\]
For the second statement we have $\mcr_{cst}(g_\dag\mco_S) \simeq a_V^+ a_{S \dag} \mco_S$, which gives
\[
\mch^{i}\mcr_{cst}(g_\dag \mco_S) \simeq H^{d+i}_c(S, \mbc) \otimes \mco_V \simeq H_{d-i}(S, \mbc) \otimes \mco_V\, ,
\]
where the last isomorphism follows from Poincar\'{e}-Verdier duality.
\item
The first statement follows from the second isomorphism in Proposition \ref{lem:AERadonFourier} and \ref{lem:Step1} (2). The second statement is again the dual of the first.
\end{enumerate}
\end{proof}

\noindent Recall the following triangles in $D^b_{rh}(V)$ from above:
\begin{align}\label{eq:Radtriangle}
\mcr(g_\dag\, \mco_S)[-1] \lra \mcr_{cst}(g_\dag\, \mco_S) \lra \mcr^\circ(g_\dag\, \mco_S) \overset{+1}{\lra}, &\\
\mcr^\circ_{c}(g_+\, \mco_S) \lra \mcr_{cst}(g_+\, \mco_S) \lra \mcr(g_+\, \mco_S)[1]   \overset{+1}{\lra}&\, . \notag
\end{align}
\noindent Using Proposition \ref{lem:Step2}, this will enable us to extract information about the cohomology of the (proper) direct image of $\varphi_{B}$.

\begin{prop}\label{prop:pervcohGM}Let $\varphi_{B}:S \times W \lra V$ be the family of Laurent polynomials defined above.
\begin{enumerate}
\item $\mch^k (\varphi_{B,+} \mco_{S \times W}) = 0$ for $k \notin [-d+1, 0]$ and $\mch^k( \varphi_{B,\dag} \mco_{S \times W}) = 0$ for $k \notin [0, d-1]$.
\item $\mch^k( \varphi_{B,+} \mco_{S \times W})$ is isomorphic to the free $\mco_V$-module  $H^{d-1+k}(S, \mbc) \otimes \mco_V$ for $k \in [-d+1, -1]$ and\\
$\mch^k( \varphi_{B,\dag} \mco_{S \times W})$ is isomorphic to the free $\mco_V$-module $H_{d-1-k}(S, \mbc) \otimes \mco_V$ for $k \in [1, d-1]$.
\item There are the following exact sequences in $M_{rh}(\mcd_V)$:
\begin{align}
&0 \lra H^{d-1}(S,\mbc)\otimes \mco_V \lra \mch^{0}( \varphi_{B,+} \mco_{S \times W}) \lra \FL(h_+ \mco_{T}) \lra H^d(S,\mbc)\otimes \mco_V \lra 0\,, \notag \\
&0 \lra H_d(S,\mbc) \otimes \mco_V \lra \FL(h_\dag \mco_T) \lra \mch^{0}(\varphi_{B,\dag} \mco_{S \times W}) \lra H_{d-1}(S,\mbc) \otimes \mco_V \lra 0\,. \notag
\end{align}
\end{enumerate}
\end{prop}
\begin{proof}
Notice that we have
\begin{align}
\mch^i \mcr^\circ_c(g_+ \mco_{S}) \simeq \mch^{i} \FL(h_+ \mco_{T}) = 0 \quad \text{for} \quad {i \neq 0}\, , \label{eq:Ropencomp} \\
\mch^i \mcr^\circ(g_\dag \mco_{S}) \simeq \mch^{i} \FL(h_\dag \mco_{T}) = 0 \quad \text{for} \quad {i \neq 0}\, , \label{eq:Ropen}
\end{align}
by Proposition \ref{lem:Step2}(3), the exactness of $h_+$ resp. $h_\dag$ (cf. Lemma \ref{lem:Step1} (1)) and the exactness of the Fourier-Laplace transformation. Additionally the following holds:
\begin{align}
\mcr(g_+ \mco_{S}) \simeq \varphi_{B,+} \mco_{S \times W} \in D^{\leq 0}_{rh}(\mcd_V)\, , \label{eq:CohomGM} \\
\mcr(g_\dag \mco_{S}) \simeq \varphi_{B,\dag} \mco_{S \times W} \in D^{\geq 0}_{rh}(\mcd_V)\, , \label{eq:CohompropGM}
\end{align}
where the isomorphisms hold by Proposition \ref{lem:Step2}(1) and the statement about the cohomology of $\varphi_{B,+} \mco_{S \times W}$ holds because $\varphi_{B}$ is affine and therefore $\varphi_{B,+}$ is right exact. The claim about the cohomology of $\varphi_{B,\dag} \mco_{S \times W}$ follows by duality.\\  

\noindent Now we take the long exact cohomology sequence of the two triangles in (\ref{eq:Radtriangle}). For the first triangle we get
\[
0 \lra \mch^{-1}(\mcr_{cst}(g_+\mco_S)) \lra \mch^{0} (\mcr(g_+\mco_S)) \lra \mch^{0}(\mcr^\circ_c(g_+\mco_S)) \lra \mch^{0}(\mcr_{cst}(g_+\mco_S)) \lra 0
\]
and
\[
\mch^{i-1}(\mcr_{cst}(g_+ \mco_S)) \simeq \mch^i(\mcr(g_+ \mco_S)) \quad \text{for} \quad i \leq -1 \, ,
\]
because of \eqref{eq:Ropencomp} and \eqref{eq:CohomGM}. For the second triangle we get
\[
0 \ra \mch^{0}(\mcr_{cst}(g_\dag \mco_S)) \ra \mch^{0}(\mcr^\circ(g_\dag \mco_S)) \ra \mch^{0}(\mcr(g_\dag \mco_S)) \ra \mch^{1}(\mcr_{cst}(g_\dag \mco_S)) \ra 0
\]
and
\[
\mch^{i}(\mcr(g_\dag \mco_S)) \simeq \mch^{i+1}(\mcr_{cst}(g_\dag \mco_S)) \quad \text{for} \quad i \geq 1\, ,
\]
because of \eqref{eq:Ropen} and \eqref{eq:CohompropGM}. Applying Proposition \ref{lem:Step2} to the single terms shows the claims.
\end{proof}

\noindent In order to relate $\FL(h_+ \mco_T)$ resp. $\FL(h_\dag \mco_T)$  to GKZ-systems, we need the following lemma.
\begin{lem}\label{lem:FLcompare}
There are the following isomorphisms in $D^b_{rh}(\mcd_V)$:
\begin{align}
\FL(h_+ \mco_T) &\simeq \bigoplus_{\gamma \in I_e} \FL(h_{\widetilde{A},+} (\mco_T \cdot \underline{y}^{(0,\gamma)}))\, , \notag \\
\FL(h_\dag \mco_T) &\simeq \bigoplus_{\gamma \in I_e} \FL(h_{\widetilde{A},\dag} (\mco_T \cdot \underline{y}^{(0,\gamma)}))\, ,  \notag
\end{align}
where 
\begin{align}
h_{\widetilde{A}}: T \lra V'\, , \notag \\
(y_0,\ldots,y_d) &\mapsto (\underline{y}^{\underline{\widetilde{a}}_0}, \ldots,\underline{y}^{\underline{\widetilde{a}}_n}) \notag
\end{align}
and $I_e$ is the set $\prod_{k=1}^d \frac{([0,e_k-1] \cap \mbn)}{e_k} \subset \mbq^d$, where $e_1, \ldots , e_d$ are the elementary divisors of the matrix $B$.
\end{lem}
\begin{proof}
First notice that the map $h$ (cf. equation \ref{def:maph}) can be factored into $h=h_{\widetilde{A}}\circ f_{D_1} \circ f_{C}$, where
\begin{align}
f_C: T &\lra T \notag \\
(y_0,\ldots ,y_d) &\mapsto (y_0,\underline{y}^{\underline{c}_1}, \ldots, \underline{y}^{\underline{c}_d}) \notag
\end{align}
\begin{align}
f_{D_1}: T &\lra T \notag \\
(y_0, \ldots ,y_d) &\mapsto (y_0, y_1^{e_1}, \ldots , y_d^{e_d}) \notag
\end{align}
and
\begin{align}
h_A : T &\lra V' \notag \\
(y_0,\ldots ,y_d) &\mapsto (y_0, y_0\underline{y}^{\underline{a}_1},\ldots,y_0\underline{y}^{\underline{a}_n}) \notag \\
&= (\underline{y}^{\underline{\widetilde{a}}_0},\ldots,\underline{y}^{\underline{\widetilde{a}}_n}) \notag
\end{align}
which corresponds to the factorization $B = C \cdot D_1 \cdot A$. Notice that $f_C$ is an isomorphism, because the matrix $C$ with columns $\underline{c}_i$ is invertible, i.e. we have $f_{C,+} \mco_T \simeq \mco_T$.  A simple computation shows that
\begin{align}
f_{D_1,+} \mco_T &\simeq \bigoplus_{\gamma \in I_e} \mcd_T / (\mcd_T y_0 \p_{y_0} + \mcd_T (y_1 \p_{y_1} - \gamma_1) + \ldots + \mcd_T (y_d \p_{y_d}- \gamma_d)) \notag \\
&= \bigoplus_{\gamma \in I_e} \mco_T\underline{y}^{(0,\gamma)}\, . \notag
\end{align}
Therefore we have
\[
h_+ \mco_T \simeq h_{\widetilde{A},+}\circ f_{D_1,+} \mco_T \simeq h_{\widetilde{A},+} (\bigoplus_{\gamma \in I_e} \mco_T\underline{y}^{(0,\gamma)}) \, .
\]
Using the additivity of the functors $h_{\widetilde{A},+}$ and $\FL$, this shows the first isomorphism. The second isomorphism will follow by duality and the fact that
\begin{align}
\mbd(\mco_T \cdot \underline{y}^{(0,\gamma)}) &= \mbd(\mcd_T / (\mcd_T y_0 \p_{y_0} + \mcd_T (y_1 \p_{y_1} - \gamma_1) + \ldots + \mcd_T (y_d \p_{y_d}- \gamma_d))) \notag \\
&\simeq \mcd_T / (\mcd_T (y_0 \p_{y_0}+1) + \mcd_T (y_1 \p_{y_1} + \gamma_1+1) + \ldots + \mcd_T (y_d \p_{y_d}+ \gamma_d+1))\, . \notag
\end{align}
If we use that $\mco_T \cdot \underline{y}^{(0,\gamma)} \simeq \mco_T \cdot \underline{y}^{(0,\gamma)+ \widetilde{\alpha}}$ for every $\widetilde{\alpha} \in \mbz^{d+1}$, we get $\mbd(\bigoplus_{\gamma \in I_e}\mco_T \cdot \underline{y}^{(0,\gamma)}) \simeq \bigoplus_{\gamma \in I_e}\mco_T \cdot \underline{y}^{(0,\gamma)}$, which shows the second claim.
\end{proof}
\begin{proof}[Proof of Theorem \ref{thm:genseq}]
By Proposition \ref{prop:pervcohGM} the only thing which we have to show for the first sequence, is the identification of $\FL(h_+ \mco_T)$  with $\bigoplus_{\widetilde{\beta} \in I} \mcm^{\widetilde{\beta}}_{\widetilde{A}}$, where $I = (\mbz \times \frac{1}{e}\mbz^{d}) \cap im(\sigma)$, $\sigma : (\mbq/\mbz)^{d+1} \ra \mbq^{d+1} \setminus sRes(\widetilde{A})$ is a section of the projection $p: \mbq^{d+1} \ra (\mbq/\mbz)^{d+1}$ and $\frac{1}{e}\mbz^d = \frac{1}{e_1}\mbz \times \ldots \times \frac{1}{e_d}\mbz$.
Notice that we have  
\[
\FL(h_+ \mco_T) \simeq \bigoplus_{\gamma \in I_e} \FL(h_{\widetilde{A},+} \mco_T \underline{y}^{(0,\gamma)}) \simeq \bigoplus_{\gamma \in I_e}\FL(h_{\widetilde{A},+} \mco_T\underline{y}^{\sigma(p(0,\gamma))})\simeq \bigoplus_{\widetilde{\beta} \in I}\FL(h_{\widetilde{A},+} \mco_T\underline{y}^{\widetilde{\beta}}) \simeq \bigoplus_{\widetilde{\beta} \in I} \mcm^{\widetilde{\beta}}_{\widetilde{A}}
\]
by Lemma \ref{lem:FLcompare}, the fact that $\mco_T\underline{y}^{(0,\gamma)} \simeq \mco_T \underline{y}^{\widetilde{\alpha} + (0,\gamma)}$ for every $\widetilde{\alpha} \in \mbz^{d+1}$ and Theorem \ref{thm:SW} (1).
This shows the existence of the first exact sequence. The second sequence follows by the same reasoning, using instead the section
$\sigma':(\mbq/\mbz)^{d+1} \ra \mbq^{d+1} \cap (\mbr \widetilde{A})^\circ$ and Theorem \ref{thm:SW} (2).
\end{proof}

\noindent If the semigroup $\mbn \widetilde{A}$ is saturated and $\mbn A = \mbz^d$ we can compute the morphism between the Gau\ss-Manin system $\mch^0(\varphi_+ \mco_{S \times W})$ and the GKZ-system $\mcm^{0}_{\widetilde{A}}$ as well as the kernel and cokernel. 

\noindent We will first deduce the description of the Gau\ss-Manin system by relative differential forms. This description is well-known to the experts but the author could not find a suitable reference. To compute $\mch^0(\varphi_+ \mco_{S \times W})$ we factor the map $\varphi$ into a closed embedding
\begin{align}
\widetilde{i}: S \times W &\lra S \times V\, , \notag \\
(\underline{y}, \underline{\lambda}) &\mapsto (\underline{y}\, , \,  F(\underline{y},\underline{\lambda})\, , \, \underline{\lambda}) = (\underline{y}, - \sum_{i=1}^n \lambda_i \underline{y}^{\underline{a}_i}, \underline{\lambda})\, \notag
\end{align}
and the projection $p:S \times V \lra V$. The image of $\widetilde{i}$ is the smooth hypersurface $\Gamma$ given by $\gamma := \lambda_0 + \sum_{i=1}^n \lambda_i \underline{y}^{\underline{a}_i} = 0$.

\noindent The direct image $\mcb_{\Gamma \mid S \times V} := \widetilde{i}_+ \mco_{S \times W}$ is isomorphic to $\mcd_{S \times V}/ I$, where the ideal $I$ is given by
{\small
\[
 I \!=\! \left(\!\mcd_{S \times V}(\p_{\lambda_i}\! - \underline{y}^{\underline{a}_i} \p_{\lambda_0})_{i=1, \ldots ,n} + \mcd_{S \times V}(\p_{y_k}\! -\!y_{k}^{-1} \sum_{i=1}^n a_{ki} \lambda_i \underline{y}^{\underline{a}_i}\p_{\lambda_0})_{k=1, \ldots ,d} + \mcd_{S \times V}(\lambda_0 \!+\!\! \sum_{i=1}^n \lambda_i \underline{y}^{\underline{a}_i})\!\right).
\]
}
Thus we can identify $\mcb_{\Gamma \mid S \times V}$ with $\mco_{S \times W}[\p_{\lambda_0}]$ which has the following action of $\mcd_{S \times V}$:
\begin{align}
\p_{\lambda_0} (g \otimes \p_{\lambda_0}^{l}) &= g \otimes \p_{\lambda_0}^{l+1}\, , \notag \\
\p_{\lambda_i} (g \otimes \p_{\lambda_0}^l) &= (\p_{\lambda_i}g)\otimes \p_{\lambda_0}^l - g (\p_{\lambda_i} F) \otimes \p_{\lambda_0}^{l+1} = (\p_{\lambda_i}g)\otimes \p_{\lambda_0}^l + g \cdot \underline{y}^{\underline{a}_i} \otimes \p_{\lambda_0}^{l+1}\, , \notag \\
\p_{y_k} (g \otimes \p_{\lambda_0}^l) &= (\p_{y_k} g) \otimes \p_{\lambda_0}^l - g\! \cdot\! (\p_{y_k} F) \otimes \p_{\lambda_0}^{l+1} = (\p_{y_k} g) \otimes \p_{\lambda_0}^l + g\! \cdot\! y_k^{-1} (\sum_{i=1}^n a_{ki} \lambda_i \underline{y}^{\underline{a}_i}) \otimes \p_{\lambda_0}^{l+1}\, , \notag \\
\lambda_0 (g \otimes \p_{\lambda_0}^l) &= g \cdot F \otimes \p_{\lambda_0}^l - l \cdot g \otimes \p_{\lambda_0}^{l-1} = g \cdot (-\sum_{i=1}^n \lambda_i \underline{y}^{\underline{a}_i}) \otimes \p_{\lambda_0}^l - l \cdot g \otimes \p_{\lambda_0}^{l-1}\, , \notag \\
\lambda_i (g \otimes \p_{\lambda_0}^l) &= \lambda_i g \otimes \p_{\lambda_0}^{l}\, , \notag \\
y_k (g \otimes \p_{\lambda_0}^{l}) &= y_k g \otimes \p_{\lambda_0}^l \, . \notag
\end{align}

\noindent The direct image under the projection $p: S \times V \lra V$ is equal to
\[
p_+ \mcb_{\Gamma \mid S \times V} \simeq Rp_* \Omega^{\bullet+d}_{S \times V /V} (\mcb_{\Gamma \mid S \times V}) \simeq Rp_* \, (\Omega^{\bullet+d}_{S} \otimes_{\mco_{S}}  \mcb_{\Gamma \mid S \times V})\, .
\]
As the map $p$ is affine, this is equal to $p_* \, (\Omega^{\bullet+d}_{S} \otimes_{\mco_{S}}  \mcb_{\Gamma \mid S \times V})$, where the differential on the last complex is given by
\[
d(\omega \otimes Q) = d \omega \otimes Q + \sum_{k=1}^d dy_k \wedge \omega \otimes \p_{y_k} Q \,.
\]
If we use the isomorphism $\mcb_{\Gamma \mid S \times V} \simeq \mco_{S \times W}[\p_{\lambda_0}]$, the latter complex becomes isomorphic to
\[
p_* \Omega^{\bullet +d}_{S \times W /W}[\p_{\lambda_0}]
\]
with differential given by
\[
d(w \otimes \p_{\lambda_0}^l )  = d \omega \otimes \p_{\lambda_0}^l - d_y F \wedge \omega  \otimes \p_{\lambda_0}^{l+1}\, .
\]
Thus the Gau\ss-Manin system $\mch^0(\varphi_+ \mco_{S \times W})$ is given by
\begin{equation}\label{eq:GMSystemwphi}
\frac{p_* \Omega^d_{S \times W /W}[\p_{\lambda_0}]}{(d - \p_{\lambda_0} d_y F \wedge)p_*\Omega^{d-1}_{S \times W /W}[\p_{\lambda_0}]}
\end{equation}
with the following $\mcd_V$-action:
\begin{align}
\p_{\lambda_0}(\omega \otimes \p_{\lambda_0}^l) &= \omega \otimes \p_{\lambda_0}^{l+1}\, , \notag \\
\p_{\lambda_i}(\omega \otimes \p_{\lambda_0}^l) &= (\p_{\lambda_i}\omega)\otimes \p_{\lambda_0}^l - \omega (\p_{\lambda_i}F) \otimes \p_{\lambda_0}^{l+1} = (\p_{\lambda_i}\omega)\otimes \p_{\lambda_0}^l + \omega \cdot \underline{y}^{\underline{a}_i} \otimes \p_{\lambda_0}^{l+1}\, , \label{eq:lambdaiaction} \\
\lambda_0 (\omega \otimes \p_{\lambda_0}^l) &= \omega \cdot F \otimes \p_{\lambda_0}^l - l \cdot \omega \otimes \p_{\lambda_0}^{l-1} = \omega \cdot (-\sum_{i=1}^n \lambda_i \underline{y}^{\underline{a}_i}) \otimes \p_{\lambda_0}^l - l \cdot \omega \otimes \p_{\lambda_0}^{l-1}\, , \label{eq:taction} \\
\lambda_i (\omega \otimes \p_{\lambda_0}^l) &= \lambda_i \omega \otimes \p_{\lambda_0}^{l} \notag
\end{align}
for $\omega \in \Omega^d_{S \times W /W}$.\\
If we use the element $\omega_0 := \frac{dy_1}{y_1} \wedge \ldots \wedge \frac{dy_d}{y_d}$ as a global section for the (locally) free sheaf $\Omega^d_{S \times W /W}$ of rank one, we get an isomorphism
\[
\mch^0(\varphi_+ \mco_{S \times W}) \simeq p_* \left( \frac{\mco_{S \times W}[\p_{\lambda_0}]}{ \,((y_k \partial_{y_k}- y_k \p_{y_k}F \p_{\lambda_0}) (\mco_{S \times W}[\p_{\lambda_0}]))_{k =1, \ldots,d}} \;\; \omega_0 \right)\,.
\]

\begin{prop}\label{prop:seqinvest}$ $\\[-15pt]
\begin{enumerate}
\item Up to multiplication with a non-zero constant the map 
\[
\psi: \mch^0(\varphi_+ \mco_{S \times W}) \lra \mcm^{(0,\underline{0})}_{\widetilde{A}}
\]
is given by
\[
\psi(\prod_{i=1}^{n} \underline{y}^{m_i \cdot \underline{a}_i} \omega_0 \otimes \p_{\lambda_0}^{s} )= \p_{\lambda_0}^{s-m+1} \p_{\lambda_1}^{m_1} \ldots \p_{\lambda_n}^{m_n}\, ,
\]
where $m = \sum_{i=1}^n m_i$.
\item The image of $\psi$ in $\mcm_{\widetilde{A}}^{(0,\underline{0})}$ is equal to the submodule generated by $\p_{\lambda_0}, \ldots , \p_{\lambda_m}$.
\item The kernel $\mbv^{n-1}$ of $\psi:\mch^0(\varphi_+ \mco_{S_0 \times W}) \lra \mcm_{\widetilde{A}}^{(0,\underline{0})}$ is spanned by $n$ flat sections given by
\[
\sum_{i=1}^m a_{ki} \lambda_i \underline{y}^{\underline{a}_i} \cdot \omega_0 \qquad \text{for} \quad k=1, \ldots ,n \, .
\]
\end{enumerate}
\end{prop}
\begin{proof}$ $\\
In the course of the proof we will use the modules of global sections instead of the $\mcd$-modules themselves.\\

First, we prove that $\psi(\omega_0) \neq 0$. As $\psi$ is not equal zero and $D_V$-linear (in particular $\mco_V$-linear), there is an element $b = \prod_{i=1}^n \underline{y}^{m_i \cdot \underline{a}_i} \omega_0 \otimes \p_{\lambda_0}^s$ with $m_i \in \mbz$ for $i = 1, \ldots , n$ such that $\psi(b) \neq 0$. 
Recall that we have $\p_{\lambda_0}^{n_0}\cdot \prod_{i=1}^n \partial_{\lambda_i}^{n_i} \cdot \omega_0 = \prod_{i=1}^n \underline{y}^{n_i \cdot \underline{a}_i} \omega_0 \otimes \p_{\lambda_0}^{\widetilde{n}}$  for $\widetilde{n} = \sum_{i=0}^n n_i$ and $n_i \in \mbn$. 
Let $I = \{ i_1, \ldots , i_r\} = \{ i \mid m_i < 0\}$, $I^c := \{ 1, \ldots n \} \setminus I$ and set $m_I:= \sum_{i \in I} (-m_i)$. 
We have $\psi(\prod_{i \in I^c} \underline{y}^{m_i \underline{a}_i} \omega_0 \otimes \p_{\lambda_0}^{s+m_I+k}) = \p_{\lambda_0}^k \prod_{i \in I} \p_{\lambda_i}^{-m_i}\psi(b)$ for every $k \geq 0$. 
Notice that $\psi(\prod_{i \in I^c} \underline{y}^{m_i \underline{a}_i} \omega_0 \otimes \p_{\lambda_0}^{s+m_I+k}) \neq 0$ because for every $j \in \{0, \ldots , n \}$ the element $\p_{\lambda_j}$ acts bijectively on $M_{\widetilde{A}}^{(0,\underline{0})}$ (cf. Theorem \ref{thm:SW}(3)).
Set $k =\max\{0, \sum_{i \in I^c} m_i -s -m_I\} = \max\{0, m-s\}$. 
The element $\prod_{i \in I^c} \underline{y}^{m_i \underline{a}_i} \omega_0 \otimes \p_{\lambda_0}^{s+m_I+k}$ can be written as $P \cdot \omega_0$ for $P =\p_{\lambda_0}^{s-m+k} \prod_{i \in I^c} \p_{\lambda_i}^{m_i} \in \mcd_V$. We conclude that $0 \neq \psi(P \cdot \omega_0) = P \cdot \psi(\omega_0)$, which shows $\psi(\omega_0) \neq 0$.\\

\noindent The element $\omega_0$ satisfies the following relations:
\begin{align}
(\lambda_0 \p_{\lambda_0} + \sum_{i=1}^n \lambda_i \partial_{\lambda_i}) \cdot \omega_0 &= -\omega_0 \, , \notag \\
\sum_{i=1}^n a_{ki} \lambda_i \partial_{\lambda_i}\, \cdot \omega_0 &= \sum_{i=1}^n a_{ki} \lambda_i \underline{y}^{\underline{a}_i} \omega_0 \otimes \p_{\lambda_0} = y_k \partial_{y_k} \cdot  \omega_0 = 0\, , \notag \\
\Box_{\underline{l}}\,\cdot \omega_0 &=  \left(\prod_{i : l_i < 0} \underline{y}^{l_i \cdot \underline{a}_i}  - \prod_{i : l_i > 0} \underline{y}^{l_i \cdot \underline{a}_i}\right) \omega_0 \otimes \p_{\lambda_0}^l = 0 \notag
\end{align}
for $l = \deg(\Box_{\underline{l}})$. This shows the existence of a morphism $M^{-1,\underline{0}}_{\widetilde{A}} \lra \Gamma(V,\mch^0(\varphi_+ \mco_{S \times W}))$ which sends $1$ to $\omega_0$. If we concatenate this morphism with the morphism
\[
\psi: \Gamma(V,\mch^0(\varphi_+ \mco_{S_0 \times W})) \lra M_{\widetilde{A}}^{0, \underline{0}}\, ,
\]
we get a non-zero morphism $M_{\widetilde{A}}^{-1, \underline{0}} \lra M_{\widetilde{A}}^{0, \underline{0}}$. Now the only non-zero $D_V$-linear morphism (up to a constant) from $M_{\widetilde{A}}^{-1, \underline{0}}$ to $M_{\widetilde{A}}^{0, \underline{0}}$ is right multiplication with $\partial_{\lambda_0}$ (cf. Proposition \ref{prop:gkzuniquemorph} ). But this shows that the image of $\omega_0$ in $M_{\widetilde{A}}^{0, \underline{0}}$ is $\partial_{\lambda_0}$ (after a possible multiplication with a non-zero constant).\\

\noindent From the discussion above we get now for some general $b = \prod_{i=1}^n \underline{y}^{m_i \cdot \underline{a}_i} \omega_0 \otimes \p_{\lambda_0}^s$ the following identities
\[
\p_{\lambda_0}^k \prod_{i \in I} \p_{\lambda_i}^{-m_i}\psi(b) = \psi(\prod_{i \in I^c} \underline{y}^{m_i \underline{a}_i} \omega_0 \otimes \p_{\lambda_0}^{s+m_I+k}) = \p_{\lambda_0}^{s-m+k} \prod_{i \in I^c} \p_{\lambda_i}^{m_i} \psi(\omega_0)
\]
Because left multiplication with respect to all $\p_{\lambda_j}$ is bijective in $M_{\widetilde{A}}^{0,\underline{0}}$ (cf. Theorem \ref{thm:SW}(3)), this gives
\[
\psi(b) = \p_{\lambda_0}^{s-m} \prod_{i=1}^n \p_{\lambda_i}^{m_i} \psi(\omega_0)\, .
\]
This shows the first point.\\ 

\noindent In particular $\p_{\lambda_1}, \ldots , \p_{\lambda_n}$ is in the image of the map $\Gamma(V,\mch^0(\varphi_+ \mco_{S \times W})) \lra M_{\widetilde{A}}^{0,\underline{0}}$. We conclude that the submodule of $M_{\widetilde{A}}^{0,\underline{0}}$ which is generated by $\p_{\lambda_0} ,\ldots , \p_{\lambda_n}$ lies in the image. Notice that $1$ does not lie in the image because otherwise the map would be surjective. But this shows that the image is in fact equal to the submodule generated by $\p_{\lambda_0} ,\ldots , \p_{\lambda_n}$ as the cokernel  $H^d(S,\mbc) \otimes \mco_V$ has no $\mco_V$-torsion. This shows the second point.\\

\noindent Consider the elements
\[
f_k := \sum_{i=1}^n a_{ki} \lambda_i \underline{y}^{\underline{a}_i} \omega_0 \qquad \text{for} \quad k=1, \ldots , d\, .
\]
Their image in $M_{\widetilde{A}}^{(0,\underline{0})}$ is equal to $\sum_{i=1}^n a_{ki} \lambda_i \p_{\lambda_i}$ which in turn is equal to $0$. Thus the $f_k$ lie in the kernel of $\psi$. It remains to show that they are flat:
\begin{align}
\p_{\lambda_0} \cdot \sum_{i=1}^n a_{ki} \lambda_i \underline{y}^{\underline{a}_i} \omega_0 &= \sum_{i=1}^n a_{ki} \lambda_i \underline{y}^{\underline{a}_i} \omega_0 \otimes \p_{\lambda_0} = y_k \p_{y_k} \cdot  \omega_0 =0 \, , \notag \\
\p_{\lambda_l} \cdot \sum_{i=1}^n a_{ki} \lambda_i \underline{y}^{\underline{a}_i} \omega_0 &= a_{kl} \underline{y}^{\underline{a}_l} \omega_0 + \left( \sum_{i=1}^n a_{ki} \lambda_i \underline{y}^{\underline{a}_i} \right) \, \underline{y}^{\underline{a}_l}\omega_0 \otimes \p_{\lambda_0} = y_k \p_{y_k} \cdot \underline{y}^{\underline{a}_l}\omega_0 = 0\, . \notag
\end{align}
This shows the third point.
\end{proof} 

\begin{rem}
Notice that the first formula in Proposition \ref{prop:seqinvest} might involve negative powers of $\p_{\lambda_j}$. By Theorem \ref{thm:SW} 3. this is well-defined, i.e. the element $\p_{\lambda_0}^{m-s+1} \p_{\lambda_1}^{m_1} \ldots \p_{\lambda_n}^{m_n}$ is the unique element $P \in M_{\widetilde{A}}^{0,\underline{0}}$ so that for $k = \max\{0 ,m-s+1\}$ we have $\p_{\lambda_0}^{k-m+s-1} \prod_{i \in I} \p_{\lambda_i}^{-m_i} \cdot P =  \p_{\lambda_0}^{k}\prod_{i \in I^c} \p_{\lambda_i}^{m_i}$. Computing this element $P$ in general seems to be difficult. Consider the GKZ-system $M_{\widetilde{A}}^{0,0}$ with
$$
\widetilde{A} = \left( \begin{matrix}1 & 1 & 1 \\ 0 & 1 & -1  \end{matrix} \right)\,.
$$
A straightforward computation shows that the element $\p_{\lambda_0}^{-1}$ in $M_{\widetilde{A}}^{0,0}$ is equal to $(\lambda_0^2 - 4 \lambda_1 \lambda_2 )\p_{\lambda_0} + \lambda_0$. One can see in this example that the expression involves the discriminant of the associated family of Laurent polynomials $(A =(1,-1))$:
\begin{align}
\varphi_A : \mbc^* \times \mbc^2 \lra \mbc_{\lambda_0} \times \mbc^2\, , \notag \\
(y, \lambda_1, \lambda_2) \mapsto (- \lambda_1 y - \lambda_2 \frac{1}{y}, \lambda_1, \lambda_2)\, .
\end{align}
\end{rem}

Up to now, only GKZ-systems $\mcm^{(\beta_0,\beta)}_{\widetilde{A}}$ with $\beta_0 \in \mbz$, occurred. We will see, that this is reflected by the fact that we looked at all fibers of the associated family of Laurent polynomials $\varphi_B$. If the matrix $A$ is homogeneous, we will remedy this fact by restricting to a hyperplane in $V = \mbc_{\lambda_0} \times W$ given by $\lambda_0 = 1$, which will give us direct sums of GKZ-systems $\mcm^\beta_A$ with $\beta \in \mbq^d$.

 In the rest of this section let $A$ be a $d \times n$ integer matrix with upper row $(1, \ldots ,1)$ which satisfies $\mbz A^d = \mbz^d$ and let $e= (e_1, \ldots ,e_d)$ with $e_i \in \mbn_{\geq 1}$. We define a matrix
\[
B:= \left( \begin{array}{c  c  c} e_1 & &  \\ & \ddots & \\ & & e_{d} \end{array} \right) \cdot A\, .
\]
From this data we will construct a family of affine varieties $p_B : \Lambda \ra W = \mbc^n$ and derive exact sequences similar to those in Theorem \ref{thm:genseq}. For this we will need the following lemma.

\begin{lem}\label{lem:invimage}
Let $i_1 : \{1\} \times W \lra V = \mbc \times W$ be the canonical inclusion. Then
\begin{enumerate}
\item The map $i_1$ is non-characteristic with respect to $\mcm_{ \widetilde{A}}^{(\beta_0, \beta)}$.
\item $i_1^+ \mcm_{ \widetilde{A}}^{(\beta_0, \beta)} \simeq \mcm_A^{\beta}\, ,$
\end{enumerate}
where $\widetilde{A}$ is given by \eqref{eq:tildeA}.
\end{lem}
\begin{proof} 
Let $Q$ be the convex hull of the columns $\underline{\widetilde{a}}_0, \ldots , \underline{\widetilde{a}}_n$ of $\widetilde{A}$ in $\mbr^{d+1}$. Denote by $\tau_1, \ldots, \tau_s$ the faces of $Q$ (including $Q$ itself). The set of singular points of $\mcm_{ \widetilde{A}}^{(\beta_0, \beta)}$ is given by the union of the hypersurfaces $V(\tau_l)$ for $l=1, \ldots ,s$ (cf. \cite[Theorem 4]{GKZ1} and \cite[Remark 10.1.8]{GKZbook}) which are given by
\[
V(\tau_l) = \{ \underline{\lambda} \in V \mid \exists \underline{y} \in S \; \text{s.t.}\; f_{\underline{\lambda},\tau_l}(\underline{y}) = y_k \p_{k}f_{\underline{\lambda},\tau_l}(\underline{y}) = 0\;\; \text{for all}\;\; k=0,\ldots,d \}\, .
\] 
with $f_{\underline{\lambda},\tau_l}(\underline{y}):= \sum_{\widetilde{\underline{a}}_i \in \tau_l} \lambda_i \underline{y}^{\widetilde{\underline{a}}_i}$ and $\underline{y}^{\widetilde{\underline{a}}_i} := \prod_{k=0}^n y_k^{\widetilde{a}_{ki}}$.

\noindent There are now two possibilities. If $\widetilde{\underline{a}}_0 \in \tau_l$, then $V(\tau_l) \subset \{\lambda_0 = 0\}$ which follows from
\[
\lambda _0 = y_0\p_0f_{\underline{\lambda},\tau_l}(\underline{y}) - y_1\p_1 f_{\underline{\lambda},\tau_l}(\underline{y}) \overset{!}{=} 0.
\]
On the other hand, if $\underline{\widetilde{a}}_0 \notin \tau_l$ then $V(\tau_l) = p^{-1}(V_l)$, where $p: V = \mbc_{\lambda_0} \times W \ra W$ is the projection and $V_l$ is a hypersurface in $W$.  This shows that the restriction to $\{1\} \times W$ is non-characteristic. Thus we have
\[
\mch^0\left(i_1^+ \mcm_{ \widetilde{A}}^{(\beta_0, \beta)}\right) \simeq i_1^+ \mcm_{ \widetilde{A}}^{(\beta_0, \beta)}\,.
\]
Recall the definition of the generators of the GKZ-system from definition \ref{def:GKZ}. Because the first row of $A$ is equal to $(1, \ldots , 1)$ all operators $\Box_{l \in \mbl}$, where $\mbl$ is the lattice of relations of the matrix $\widetilde{A}$ are independent of $\p_{\lambda_0}$. Notice also that all Euler vector fields, except $E_0$, are independent of $\lambda_0 \p_{\lambda_0}$.
Working with the $D$-module of global sections instead of the actual $\mcd$-module, the inverse image can be written as
\[
M_{\widetilde{A}}^{(\beta_0, \beta)} / (\lambda_0-1) M_{\widetilde{A}}^{(\beta_0, \beta)}\,.
\]
As a $D_W$ module this is isomorphic to
\[
\mbc[\lambda_1, \ldots , \lambda_n]\langle \p_{\lambda_0}, \ldots ,\p_{\lambda_n}\rangle/ I \, 
\]
where the ideal $I$ is generated by the Euler fields $E_1, \ldots , E_d$, the box operators $\Box_{l \in \mbl}$ and the operator
\[
\p_{\lambda_0} + \sum_{i=1}^n \lambda_i \p_{\lambda_i}.
\]
But this module is isomorphic to $M_A^\beta$, which shows the claim.
\end{proof}

We now define the restriction of the family of Laurent polynomials in order to get a family of affine varieties whose Gau\ss-Manin system will be closely related to a direct sum of GKZ-system $\mcm_A^\beta$. Let
\[
\Lambda := \{(y_1, \ldots,y_d,\lambda_1, \ldots ,\lambda_d \in (S \times W) \mid F_B(\underline{y},\underline{\lambda})= 1 \}\, ,
\]
where $F_B: S\times W \ra \mbc_{\lambda_0} \times W$ is the first component of $\varphi_B$. Denote by $p_B: \Lambda \ra W$ the projection to the second factor. 
\begin{thm}\label{thm:genseqres}
Let $B$ and $A$ be as above and let $p_B: \Lambda \lra W$ be the corresponding family of affine varieties. Let 
\begin{align}
\sigma : (\mbq / \mbz)^{d} &\ra \mbq^{d}\setminus sRes(A)\, , \notag \\
\sigma': (\mbq / \mbz)^{d} &\ra \mbq^{d} \setminus \mbd sRes(A) \notag
\end{align}
be sections of the projection $p: \mbq^{d} \ra (\mbq/\mbz)^{d}$ and let $I := \frac{1}{e}\mbz \cap im(\sigma)$ resp. $I' := \frac{1}{e}\mbz \cap im(\sigma')$. Then we have the following exact sequences in $M_{rh}(\mcd_W)$:
\begin{align}
&0 \lra H^{d-1}(S,\mbc)\otimes \mco_W \lra \mch^{0}(p_{B,+} \mco_{\Lambda}) \lra \bigoplus_{\beta \in I} \mcm^{\beta}_{A} \lra H^{d}(S,\mbc)\otimes \mco_W \lra 0\, , \notag \\
&0 \lra H_{d}(S,\mbc)\otimes \mco_W \lra \bigoplus_{\beta' \in I'} \mcm^{\beta'}_{A}  \lra \mch^{0}(p_{B,\dag} \mco_\Lambda) \lra H_{d-1}(S,\mbc)\otimes \mco_W \lra 0\, . \notag
\end{align}
\end{thm} 
\begin{proof}
First notice that we can lift the sections $\sigma$, $\sigma'$ to sections
\begin{align}
\widetilde{\sigma} : (\mbq / \mbz)^{d+1} &\ra \mbq^{d+1}\setminus sRes(\widetilde{A})\, , \notag \\
\widetilde{\sigma}': (\mbq / \mbz)^{d+1} &\ra \mbq^{d+1} \setminus \mbd sRes(\widetilde{A}) \notag
\end{align}
by Lemma \ref{lem:betalift} resp. by the definition of $\mbd sRes(A)$. By Theorem \ref{thm:genseq} we have exact sequences in $M_{rh}(\mcd_V)$:
\begin{align}
&0 \lra H^{d-1}(S,\mbc)\otimes \mco_V \lra \mch^{0}(\varphi_{B,+} \mco_{S \times W}) \lra \bigoplus_{\widetilde{\beta} \in \widetilde{I}} \mcm^{\widetilde{\beta}}_{\widetilde{A}} \lra H^{d}(S,\mbc)\otimes \mco_V \lra 0\, , \notag \\
&0 \lra H_{d}(S,\mbc)\otimes \mco_V \lra \bigoplus_{\widetilde{\beta}' \in \widetilde{I}'} \mcm^{\widetilde{\beta}'}_{\widetilde{A}}  \lra \mch^{0}(\varphi_{B,\dag} \mco_{S \times W}) \lra H_{d-1}(S,\mbc)\otimes \mco_V \lra 0\, , \notag
\end{align}
with $\widetilde{I} := (\mbz \times \frac{1}{e}\mbz) \cap im(\widetilde{\sigma})$ and $\widetilde{I}' := \mbz \times \frac{1}{e}\mbz \cap im(\widetilde{\sigma}')$. Notice that we have $Char(\mch^{0}(\varphi_{B,+} \mco_{S \times W}))=Char(\bigoplus_{\widetilde{\beta} \in \widetilde{I}}\mcm^{\widetilde{\beta}}_{\widetilde{A}})$ resp. $Char(\bigoplus_{\widetilde{\beta}' \in \widetilde{I}'} \mcm^{\widetilde{\beta}'}_{\widetilde{A}})=Char(\mch^{0}(\varphi_{B,\dag} \mco_{S \times W}))$. Therefore by Lemma \ref{lem:invimage} (1) the map $i_1$ is non-characteristic with respect to all terms above. Recall that we have $i_1^+ (\bigoplus_{\widetilde{\beta} \in \widetilde{I}}\mcm^{\widetilde{\beta}}_{\widetilde{A}}) \simeq \bigoplus_{\beta \in I}\mcm^{\beta}_{A}$ resp. $i_1^+ (\bigoplus_{\widetilde{\beta}' \in \widetilde{I}'}\mcm^{\widetilde{\beta}'}_{\widetilde{A}}) \simeq \bigoplus_{\beta' \in I'}\mcm^{\beta'}_{A}$ by Lemma \ref{lem:invimage} (2). Notice that we have the following cartesian diagram
\[
\xymatrix{\Lambda \ar[d]_{p_B} \ar[r] & S \times W \ar[d]_{\varphi_B} \\ \{1\} \times W \ar[r]^{i_1} & V}
\]
Using base change with respect to $i_1$ we get $i_1^+ \mch^0(\varphi_{B,+}\mco_{S \times W}) \simeq \mch^0(p_{B,+} \mco_\Lambda)$ resp.\\ $i_1^+ \mch^0(\varphi_{B,\dag}\mco_{S \times W}) \simeq \mch^0(p_{B,\dag} \mco_\Lambda)$. This shows the claim.
\end{proof}

\begin{rem}
Restricting the first exact sequence of Theorem \ref{thm:genseqres} to a generic point $\underline{\lambda} \in W$ gives us the exact sequence of mixed Hodge structures:
\[
0 \lra H^{d-1}(S,\mbc) \lra H^{d-1}(F_B^{-1}(1,\underline{\lambda}),\mbc) \lra H^d(S, F_B^{-1}(1,\underline{\lambda}),\mbc) \lra H^{d}(S,\mbc) \lra 0\, ,
\]
which is equation $(55)$ from \cite{Sti}. Setting $B=A$ and  $\beta=0$, this recovers Theorem 8 of \cite{Sti}, which says that the GKZ-system $\mcm^0_A$ restricted to its smooth locus is isomorphic to the cohomology bundle $H^d(S, F_B^{-1}(1,\underline{\lambda}),\mbc)$.
\end{rem}
\newpage
\section{Hypergeometric systems and Mixed Hodge Modules}\label{sec:MHM}
In this section we show that we can endow a GKZ-hypergeometric system with integer parameter with a structure of a mixed Hodge module in the sens of \cite{SaitoMHM}. First we show that the exact sequences in Theorem \ref{thm:genseq} resp. Corollary \ref{cor:exseqA} are actually exact sequences in the category of mixed Hodge modules. For this we have to translate their proofs into this category.

\noindent With regard to Theorem \ref{thm:genseq} resp. Corollary \ref{cor:exseqA} this might be expected as the other three terms of the exact sequences carry a natural structure of a mixed Hodge module (the two outer terms are actually (constant) variations of mixed Hodge structures). However we can not conclude directly that the (direct sum of) GKZ-systems carry a mixed Hodge module structure because the category of mixed Hodge modules is not stable by extension.

For a smooth algebraic variety $X$, we denote by $MHM(X)$ the abelian category of mixed Hodge modules and by $D^b MHM(X)$ the corresponding bounded derived category. The forgetful functors to the bounded derived category of algebraic, constructible sheaves of $\mbq$-vector spaces resp. regular holonomic $\mcd$-modules are denoted by 
\[
rat : D^b MHM(X) \lra D^b_c(X,\mbq)
\]
resp.
\[
Dmod: D^b MHM(X) \lra D^b_{rh}(\mcd_X)
\]
\noindent For each morphism $f:X \ra Y$ between complex algebraic varieties, there are induced functors $f_*,f_!: D^b MHM(X) \ra D^b MHM(Y)$, $f^*, f^!: D^b MHM(Y) \ra D^b MHM(X)$ which are interchanged by $\mbd$ and which lift the analogous functors $f_+, f_\dag, f^\dag-[d],  f^+[d]$ on $D^b_{rh}(\mcd_X)$ resp. $Rf_*, f_!, f^{-1}, f^!$ on $D^b_c(X)$, where $d:= \dim X -\dim Y$.

\noindent Let $\mbq^{H}_{pt}$ be the unique mixed Hodge structure of weight $0$ with $Gr^W_i = 0$ for $i \neq 0$ and underlying vectorspace $\mbq$. Denote by $a_{X} : X \lra \{pt\}$ the map to the point and set $\mbq^H_X := a_X^* \mbq^H_{pt}$.

\noindent Recall that by \cite{SaitoMHM} (4.4.3) a base change theorem holds also in the category of algebraic mixed Hodge modules.
 
\noindent Notice that the various functors $\mcr, \mcr_{cst}, \mcr^\circ_{(c)}$ are just a concatenation of (proper) direct image functors and (exceptional) inverse image functors, which means they are also defined in the derived category of (algebraic) mixed Hodge modules. Define the following functors from $D^b MHM(\mbp(V'))$ to $D^b MHM(V)$ by
\begin{align}
{^*}\mcr(M) &:= \pi_{2*}^Z (\pi_1^Z)^*M \simeq \pi_{2 *} i_{Z *} i_Z^* \pi_1^*M\, , \notag \\
{^*}\mcr_{cst}(M) &:= \pi_{2*} \pi_1^*M\, . \notag
\end{align}
Notice that unlike in the category of $\mcd$-modules ${^*}\mcr$ and ${^*}\mcr_{cst}$ commute with the duality functor $\mbd$ only up to shift and Tate twist. We therefore define the following functors from $D^b MHM(\mbp(V'))$ to $D^b MHM(V)$:
\begin{align}
{^!}\mcr(M) &:= \mbd \circ {^*}\mcr \circ \mbd \, (M) \simeq \pi_{2*}^Z (\pi_1^Z)^! M \simeq \pi_{2 *} i_{Z *} i_Z^! \pi_1^! M \notag \\
{^!}\mcr_{cst}(M) &:= \mbd \circ {^*}\mcr_{cst} \circ \mbd\,(M)\simeq  \pi_{2*} \pi_1^! M\, . \notag
\end{align}
Finally we define ${^*}\mcr^\circ_c : D^b MHM(\mbp(V')) \ra D^b MHM(V)$ by
\[
{^*}\mcr^\circ_c(M) := \pi_{2 !}^U (\pi_1^U)^*(M) \simeq \pi_{2 *} j_{U !} j_U^* \pi_1^* (M)
\]
and ${^!}\mcr^\circ : D^b MHM(\mbp(V') \ra D^b MHM(V)$ by 
\[
{^!}\mcr^\circ(M):= \mbd \circ {^*}\mcr^\circ_c \circ \mbd\, (M) \simeq  \pi_{2 *}^U (\pi_1^U)^!(M) \simeq \pi_{2 *} j_{U *} j_U^! \pi_1^! (M)\, .
\]
Using these definitions we get the triangles equivalent to Proposition \ref{prop:Radontriangle}.
\begin{prop}
Let $M \in  D^b MHM (\mbp(V'))$, we have the following triangles
\begin{align}
{^!}\mcr(M) \lra {^!}\mcr_{cst}(M) \lra {^!}\mcr^\circ(M) \overset{+1}{\lra}\, , \notag \\
{^*}\mcr^\circ_{c}(M) \lra {^*}\mcr_{cst}(M) \lra {^*}\mcr(M) \overset{+1}{\lra}\, , \notag
\end{align}
where the second triangle is dual to the first.
\end{prop}
\begin{proof}
The proof is the same as in Proposition \ref{prop:Radontriangle} using \cite[(4.4.1)]{SaitoMHM}.
\end{proof}

\begin{defn}\label{def:HodgeGKZ}
Let $\widetilde{A}$, $I$ and $I'$ be as in Theorem \ref{thm:genseq}. Define the following objects in $MHM(V)$:
\begin{align}
(\,\bigoplus_{\widetilde{\beta} \in I} \mcm^{\widetilde{\beta}}_{\widetilde{A}}\,)^H &:= \mch^{d+n+1}( {^*}\mcr^{\circ}_c(g_* \mbq^H_S)) \notag \\
(\,\bigoplus_{\widetilde{\beta} \in I'} \mcm^{-\widetilde{\beta}'}_{\widetilde{A}})^H &:= \mch^{-d-n-1}\!( {^!}\mcr^{\circ}_c(g_! \mbd\mbq^H_S)) \notag
\end{align}
\end{defn}

\begin{prop}\label{prop:MHMgenseq}
Let $B$, $\widetilde{A}$, $I$ and $I'$ be as in Theorem \ref{thm:genseq}. We have the following exact sequences in $MHM(V)$:
{\small
\begin{align}
&0 \ra \mch^{n+1}(H^{d-1}(S,\mbc)\otimes \mbq^H_V) \ra \mch^{d+n}(\varphi_{B,*} \mbq_{S \times W}^H) \ra (\bigoplus_{\widetilde{\beta} \in I} \mcm^{\widetilde{\beta}}_{\widetilde{A}})^H \lra \mch^{n+1}(H^{d}(S,\mbc)\otimes \mbq^H_V) \ra 0\, , \notag \\
&0\! \ra\! \mch^{-n-1}( H_{d}(S,\mbc)\otimes \mbd\mbq^H_V)\! \ra \!\!(\bigoplus_{\widetilde{\beta} \in I'} \mcm^{-\widetilde{\beta}'}_{\widetilde{A}})^H\!  \ra\! \mch^{-d-n}(\varphi_{B,!} \mbd\mbq^H_{S \times W})\! \ra\! \mch^{-n-1}(H_{d-1}(S,\mbc)\otimes \mbd\mbq^H_V)\! \ra\! 0\, . \notag
\end{align}
} In particular we have
\begin{align}
Dmod( (\,\bigoplus_{\widetilde{\beta} \in I} \mcm^{\widetilde{\beta}}_{\widetilde{A}}\,)^H ) &\simeq (\,\bigoplus_{\widetilde{\beta} \in I} \mcm^{\widetilde{\beta}}_{\widetilde{A}}\,)\, , \notag \\
Dmod ((\,\bigoplus_{\widetilde{\beta} \in I'} \mcm^{-\widetilde{\beta}'}_{\widetilde{A}})^H) &\simeq (\,\bigoplus_{\widetilde{\beta} \in I'} \mcm^{-\widetilde{\beta}'}_{\widetilde{A}})\, . \notag
\end{align}
\end{prop}

\begin{proof}
Recall that we derived the first exact sequence of Theorem \ref{thm:genseq} by taking the long exact cohomology sequence of the second triangle in \eqref{eq:Radtriangle}:
{
$$
\xymatrix{H^{d-1}(S,\mbc)\otimes \mco_V \ar[r]& \mch^0(\varphi_{B,+} \mco_{S \times W}) \ar[r] & \bigoplus_{\widetilde{\beta} \in I} \mcm^{\widetilde{\beta}}_{\widetilde{A}} \ar[r] & H^d(S,\mbc)\otimes \mco_V \\ \mch^{-1}( \mcr_{cst}(g_+ \mco_S)) \ar[u]^{\simeq} \ar[r] & \mch^{0}( \mcr(g_+ \mco_S)) \ar[r] \ar[u]^{\simeq} & \mch^{0}( \mcr^{\circ}_c(g_+ \mco_S)) \ar[r] \ar[u]^{\simeq} & \mch^{0}( \mcr_{cst}(g_+ \mco_S)) \ar[u]^{\simeq} }
$$
}
\noindent where we used Proposition \ref{lem:Step2} 1., 2. for the first, second and fourth isomorphism. For the third isomorphism we used the following isomorphisms
\[
\bigoplus_{\widetilde{\beta} \in I} \mcm_{\widetilde{A}}^{\widetilde{\beta}} \simeq \textup{FL}(h_+ \mco_{T}) \simeq \mcr^\circ_c(g_+ \mco_{S})\, ,
\] 

\noindent To show that the lower sequence is a sequence of mixed Hodge modules we replace the $\mcd$-module $\mco_S$ with the mixed Hodge module $\mbq^{H}_{S} := a_{S}^{*} \mbq^{H}_{pt}$ and apply the corresponding functors in the (derived) category of mixed Hodge modules. Notice that there is a subtle point. In Saito's theory $\mbq^H_X$ lies in degree $\dim X$ and for $f: X \ra Y$ the functors $f^*, f^!$ correspond to the functors $f^\dag[-\dim X + \dim Y]$ resp. $f^+[\dim X - \dim Y]$ on the level of $\mcd$-modules.
If we translate the proofs above into the category of mixed Hodge modules and take these shifts into account we get
{\small
$$
\xymatrix@C=12pt{\mch^{n+1}(H^{d-1}(S,\mbc)\otimes \mbq^H_V) \ar[r]& \mch^{d+n}(\varphi_{B,*} \mbq^H_{S \times W}) \ar[r] & \bigoplus_{\widetilde{\beta} \in I} \mcm^{\widetilde{\beta}}_{\widetilde{A}} \ar[r] & \mch^{n+1}(H^d(S,\mbc)\otimes \mbq^H_V) \\ 
\mch^{d+n}( {^*}\mcr_{cst}(g_* \mbq^H_S)) \ar[u]^{\simeq} \ar[r] & \mch^{d+n}( {^*}\mcr(g_* \mbq^H_S)) \ar[r] \ar[u]^{\simeq} & \mch^{d+n+1}( {^*}\mcr^{\circ}_c(g_* \mbq^H_S)) \ar[r] \ar[u]^{\simeq} & \mch^{d+n+1}( {^*}\mcr_{cst}(g_* \mbq^H_S)) \ar[u]^{\simeq} }
$$
}
The lower sequence is an exact sequence of mixed Hodge modules by construction. If we induce the mixed Hodge module structure of $\mch^{d+n+1}( {^*}\mcr^{\circ}_c(g_* \mbq^H_S))$ on $\bigoplus_{\widetilde{\beta} \in I} \mcm^{\widetilde{\beta}}$  the upper sequence becomes a sequence of mixed Hodge modules, too. The statement for the second sequence follows if we dualize the two sequences above:
{\small
$$
\xymatrix@C=6pt{\mch^{-n-1}(H_{d}(S,\mbc)\!\otimes\! \mbd\mbq^H_V) \ar[r]& \bigoplus_{\widetilde{\beta}' \in I'} \mcm^{-\widetilde{\beta}'}_{\widetilde{A}} \ar[r] & \mch^{-d-n}(\varphi_{B,!} \mbd\mbq^H_{S \times W}) \ar[r] & \mch^{-n-1}\!(H_{d-1}(S,\mbc)\!\otimes\! \mbd\mbq^H_V) \\ 
\mch^{-d-n-1}\!( {^!}\mcr_{cst}(g_! \mbd\mbq^H_S)) \ar[u]^{\simeq} \ar[r] & \mch^{-d-n-1}\!( {^!}\mcr^{\circ}_c(g_! \mbd\mbq^H_S)) \ar[r] \ar[u]^{\simeq} & \mch^{-d-n}( {^!}\mcr(g_! \mbd\mbq^H_S))  \ar[r] \ar[u]^{\simeq} & \mch^{-d-n}( {^!}\mcr_{cst}(g_! \mbd\mbq^H_S)), \ar[u]^{\simeq} }
$$
}
where we have used that $\mch^j(a_{S *}a_S^* \mbq^H_{pt}) \simeq H^j(S, \mbc)$ and $\mch^j(a_{S !}a_S^! \mbq^H_{pt}) \simeq H_{-j}(S, \mbc)$ as an isomorphism of mixed Hodge structures.
\end{proof}

\begin{prop}\label{prop:MHMexseqA}
Let $A$ be an integer $d \times n$-matrix with $\mbz A = \mbz^d$.  For every $\widetilde{\beta}, \widetilde{\beta}' \in \mbz^{d+1}$ with $\widetilde{\beta} \notin sRes(A)$ resp. $\widetilde{\beta}' \notin \mbd sRes(A)$ we have the following exact sequences in $MHM(V)$:
{\small
\[ 
0 \ra \mch^{n+1}(H^{d-1}(S,\mbc)\otimes \mbq^H_V) \lra \mch^{d+n}(\varphi_{A,*} \mbq^H_{S \times W}) \lra  (\mcm^{\widetilde{\beta}}_{\widetilde{A}})^H \lra \mch^{n+1}(H^{d}(S,\mbc)\otimes \mbq^H_V) \ra 0\, .
\]
\[
0 \ra \mch^{-n-1}( H_{d}(S,\mbc)\otimes \mbd\mbq^H_V) \ra  (\mcm^{-\widetilde{\beta}'}_{\widetilde{A}})^H  \ra \mch^{-d-n}(\varphi_{B,!} \mbd\mbq^H_{S \times W}) \ra \mch^{-n-1}(H_{d-1}(S,\mbc)\otimes \mbq^H_V) \ra 0\, .
\]
}
\end{prop}
\begin{proof}
The proof of this proposition is parallel to the proof of Proposition \ref{prop:MHMgenseq}. One just has to use Corollary \ref{cor:exseqA} instead of Theorem \ref{thm:genseq}.
\end{proof}

We are finally able to proof the main results of this section. Let $A$ be a $d \times n$ integer matrix with columns $\underline{a}_1, \ldots , \underline{a}_n$ so that there exists a linear function $h: \mbz^d \ra \mbz$ satisfying $h(\underline{a}_i)=1$ for all $i$. Recall that a GKZ-system corresponding to this matrix is called homogeneous. Schulze and Walther have shown in \cite{SchulWalth1} that a GKZ-system is regular holonomic if and only if it is homogeneous. 
\begin{thm}\label{thm:MHM} 
The homogeneous GKZ-system $\mcm_{A}^{\beta}$ carries a mixed Hodge module structure if $\beta \in \mbz^d$ and if one of the following conditions are satisfied
\begin{enumerate}
\item $\beta \notin sRes(A)$,
\item $\beta \notin \mbd sRes(A)$.
\end{enumerate}
\end{thm}
\begin{proof}
Let $l: \mbz^d \ra \mbz^d$ be an isomorphism, $L$ be the corresponding invertible integer matrix and $l_\mbc :\mbc^d \ra \mbc^d$ its $\mbc$-linear extension. First notice, that for $\check{\beta}:= l(\beta)$ and $\check{A}:= L \cdot A$ the GKZ-systems $\mcm^{\check{\beta}}_{\check{A}}$ and $\mcm^\beta_A$ are isomorphic. Furthermore we have $l_\mbc(sRes(A)) = sRes(\check{A})$ because we have $l(deg(S_{A})) = deg(S_{\check{A}})$ and $l(\mbd sRes(A) \cap \mbz^d) =\mbd sRes(\check{A}) \cap \mbz^d$. It is easy to see that for any $A$ there exists an $l$ resp. $L$ such that $L \cdot A$ is a $d \times n$ matrix with first row equal to $(1, \ldots ,1)$. But for the corresponding GKZ-system $\mcm^{\check{\beta}}_{\check{A}}$ the proof of Proposition \ref{prop:MHMexseqA} shows that it carries a mixed Hodge module structure.
\end{proof}
\begin{thm}
The homogeneous GKZ-system $\mcm^{\beta}_{A}$ has quasi-unipotent monodromy if $\beta \in \mbq^d$ and if one of the following conditions are satisfied
\begin{enumerate}
\item $\beta \notin sRes(A)$,
\item $\beta \notin \mbd sRes(A)$.
\end{enumerate}
\end{thm}
\begin{proof} 
As in the proof of Theorem \ref{thm:MHM} we can reduce to the case where the matrix $A$ of the GKZ-system has $(1,\ldots,1)$ as its first row. Let $e \in \mbn$ such that $e \cdot \beta \in \mbz^{d}$ and set $B := e\cdot I_{d\times d} \cdot A$. Assume $\beta \notin sRes(A)$. Using Lemma \ref{lem:sResnonnormal} we can find a section $\sigma: (\mbq/\mbz)^d \ra \mbq^{d} \setminus sRes(A)$ with $\beta \in im(\sigma)$. This section lifts to a section $\widetilde{\sigma}: (\mbq/\mbz)^{d+1} \ra \mbq^{d+1} \setminus sRes(\widetilde{A})$ by Lemma \ref{lem:betalift}. Applying $i_1^+$ to the underlying $\mcd$-module of the third term of the corresponding exact sequence in Proposition \ref{prop:MHMgenseq}, we get the following isomorphisms
\[
\bigoplus_{\beta' \in I} \mcm^{\beta'}_A \simeq i_1^+ \bigoplus_{\widetilde{\beta}' \in \widetilde{I}} \mcm^{\widetilde{\beta}'}_{\widetilde{A}} \simeq i_1^+ Dmod(\mch^{d+n+1}( {^*}\mcr^{\circ}_c(g_* \mbq^H_S))) \simeq Dmod \mch^{d+n+1}(i_1^*( {^*}\mcr^{\circ}_c(g_* \mbq^H_S)))
\]
which shows that $\mcm^{\beta}_{A}$ is isomorphic to a direct summand of a mixed Hodge module $\mcn$.

There exists a stratification $\mcs = \{S_i\}$ such that the restriction to $S_i \setminus S_{i+1}$ is a smooth mixed Hodge module, i.e. it is a polarizable variation of mixed Hodge structures. Now it follows from standard Hodge theory that the underlying local system of this restriction has (local) quasi-unipotent monodromy in the sense of \cite{Ka3}. But $\mcm_{A}^{\beta}$ is a direct summand of $Dmod(\mcn)$ from which follows that $DR(\mcm_{A}^{\beta})$ is a direct summand in $rat(\mcn) \otimes \mbc$. But this shows the first claim. The proof of the second claim is similar.
\end{proof}

\begin{rem}
If $A$ satisfies $\mbz A = \mbz^d$ one can easily show  by mimicking the proof of Theorem \ref{thm:genseq} resp. Theorem \ref{thm:genseqres}, that for $\beta \notin sRes(A)$ we have the isomorphism
\begin{equation}\label{eq:genbeta}
\mcm_A^\beta \simeq i_1^+ \mcr^\circ_c(g_+ \mco_S \cdot \underline{y}^\beta)
\end{equation}
Since the the underlying local system $\mcl_\beta$ of the $\mcd$-module $\mco_S \cdot \underline{y}^\beta$ has no $\mbq$-structure for general $\beta$, we can not endow it with a mixed Hodge module structure.  However, the assumption of the existence of a $\mbq$-structure can be relaxed. Schmid and Vilonen sketched in \cite{SchVil} the construction of a category of so-called complex (algebraic) mixed Hodge modules which do not rely on the existence of an underlying $\mbq$-structure. In order to be able to define a conjugate Hodge filtration they need however a polarization. Notice that for $\beta \in \mbr^d$  the rank one local system $\mcl_\beta$ has monodromy eigenvalues with absolute value one, hence it can be equipped with a hermitian pairing.  Thus the $\mcd$-module $\mco_S\cdot \underline{y}^{\beta}$ can be equipped with the structure of a complex mixed Hodge module in the sense of loc. cit. . Using the stability of the derived category of complex mixed Hodge modules under the (proper) direct image functors we can conclude by \eqref{eq:genbeta} that $\mcm^\beta_A$ carries the structure of a complex mixed Hodge module for $\beta \in \mbr^d$ and $\beta \notin sRes(A)$ (equivalently for $\beta \notin \mbd sRes(A))$. Furthermore, this shows that the mixed Hodge modules appearing in Definition \ref{def:HodgeGKZ} split into complex mixed Hodge modules corresponding to the direct sum of GKZ-systems.

\end{rem}

\bibliographystyle{amsalpha}
\bibliography{/home/thomas/Dropbox/Bibtex/My.bib}
$ $\\[25pt]
\noindent
Thomas Reichelt\\
Lehrstuhl f\"ur Mathematik VI\\
Universit\"at Mannheim\\
68131 Mannheim\\
Germany\\
Thomas.Reichelt@math.uni-mannheim.de
\end{document}